\documentclass{amsart}
\usepackage{latexsym,amscd,amssymb,amsopn,amsthm,amsfonts,amsmath}
\usepackage{color}
\usepackage[dvipsnames]{xcolor} 
\usepackage{mathrsfs}
\copyrightinfo{2010}{American Mathematical Society}
\newtheorem{theorem}{Theorem}[section]
\newtheorem{lemma}[theorem]{Lemma}
\newtheorem{corollary}[theorem]{Corollary}

\theoremstyle{definition}

\newtheorem{example}[theorem]{Example}
\newtheorem{problem}[theorem]{Problem}

\theoremstyle{remark}
\newtheorem{remark}[theorem]{Remark}
\numberwithin{equation}{section}

\begin{document}

\title[On the uniqueness theorems for transmission problems]
{On the uniqueness theorems for transmissions problems related to  
models of elasticity, diffusion and electrocardiography}
\author[A. Shlapunov]{Alexander Shlapunov}

\email{ashlapunov@sfu-kras.ru}

\author[Yu. Shefer]{Yulia Shefer}
\email{yushefer@mail.ru}
\address{Siberian Federal University,
         Institute of Mathematics and Computer Science,
         pr. Svobodnyi 79,
         660041 Krasnoyarsk,
         Russia}

\subjclass [2010] {Primary 35J56; Secondary 35J57, 35K40}

\keywords{transmission problems, uniqueness theorems, the   
inverse problem of electrocardiography}

\begin{abstract}
We consider a generalization of the inverse problem of the electrocardiography in the 
framework of the theory of elliptic and parabolic 
differential operators. More precisely, starting with 
the standard bidomain mathematical model  related to the problem of  
the reconstruction of the transmembrane potential in the myocardium from known 
body surface potentials we formulate a more general transmission problem for elliptic and 
parabolic equations in the Sobolev type spaces and describe conditions, providing 
uniqueness theorems  for its solutions. Next, the new transmission 
problem is interpreted in the framework of the elasticity theory applied to composite 
media.  Finally, we prove a uniqueness theorem for an evolutionary 
transmission problem that  can be easily adopted 
to many models involving the diffusion type equations. 
\end{abstract}

\maketitle

\section*{Introduction}
\label{s.0}

Transmission problems for differential equations appear  in many applications, 
see, for instance, \cite{Bor10} in relation to the elliptic theory or \cite{eid} for 
parabolic operators. One of the 
topical example is the inverse problem of electrocardiography. It is 
the problem of (numerical) 
reconstruction of cardiac electrical activity from ECG measurements on the body 
surface having  a significant value for diagnostics and treatment of cardiac arrhythmias, 
see \cite{1}, \cite{geselowitz1983},  \cite{2} and elsewhere. The problem 
involves several boundary value problems for elliptic and parabolic 
differential operators. First, it is  Cauchy problem for elliptic operators 
that can be treated in the framework of the theory of the ill-posed problems, see
 \cite{Lv1}, \cite{KMF91}, \cite{Tark36}, \cite{TikhArsX}. Second, these are 
the Dirichlet problem and the Neumann problem for strongly elliptic operators 
possessing the Fredholm property (see, for instance, \cite{GiTru83}, 
\cite{McL00}, \cite{Mikh}, \cite{Roit96}   and \cite{Simanca1987} for their 
treatment in various function spaces). 
The model contains also an evolutionary part, see \cite{AP96}, 
\cite{2}, involving rather general non-linear parabolic equations,  that 
in some particular cases can be treated by the classical methods, 
see, for instance, \cite{LadSoUr67}, \cite{Lion69}. 

Recently,  theoretical investigations of the steady part of the model led to interesting 
results about non-uniqueness and existence of its solutions in 
Hardy type spaces, see \cite{2K}. Paper 
\cite{Shef} was devoted to a larger class of similar transmission 
problems in the Sobolev spaces in the framework of general theory of 
elliptic operators with constant coefficients. 
However the both results were obtained under the following 
very restrictive assumptions: all the elliptic operators involved in the model 
should be proportional. 

In the present work, we aim to describe conditions, providing uniqueness theorems  
for essentially general transmission problems of this kind involving 
both elliptic and parabolic differential operators. As an example,  the new steady 
transmission problem was interpreted in the framework of the elasticity theory applied to 
composite media. An uniqueness theorem was also proved for an evolutionary 
transmission problem that can be interpreted in the framework of the theory 
of diffusion processes. 

\section{The  bidomain model of the electrocardiography}
\label{s.bidomain.steady}

Let $\theta$ be a measurable set in ${\mathbb R}^n$, $n\geq 2$ (of course, 
for models of cardiology we may always restrict ourselves to $n=3$). 

Denote by $L^2(\theta)$ a Lebesgue space of functions on  $\theta$ with the 
inner product
\begin{equation*}
\left( u,v \right)_{L^2(\theta)} = \int_{\theta} v (x)u (x)\ dx .
\end{equation*}
If $D$ is a domain (an open connected set) in ${\mathbb R}^n$ with a piecewise smooth boundary
$\partial D$, then for $s \in \mathbb{N}$ we denote by $H^s(D)$ 
the standard Sobolev space with the inner product
\begin{equation*}
\left( u,v \right)_{H^s(D)} = \int_{D}\sum_{|\alpha| \le s} 
(\partial^{\alpha}v) (\partial^{\alpha}u)
dx .
\end{equation*}
As usual, denote by $H^s_0 (D)$ the closure of the subspace $C^{\infty}_{0} (D)$ in 
$H^{s} (D)$, where $C^{\infty}_{0} (D)$ is the linear space of functions with 
compact supports in  $D$. 

Next, given any non-integer $s > 0$, let $H^s (D)$ stand for 
the so-called Sobolev-Slobodetskii space. It can 
be defined as the completion of $C^{\infty} (\overline{D})$ with respect to the norm
$$
   \| u \|_{H^s (D)}
 = \Big( \| u \|^2_{H^{[s]} (D)}
       + \int \!\!\! \int_{D \times D}
         \sum_{|\alpha | = [s]}
         \frac{|\partial^\alpha u (x) - \partial^\alpha u (y)|^2}{|x - y|^{n+2(s-[s])}}\,
         dx dy
   \Big)^{1/2},
$$
where $[s]$ is the integer part of $s$, see \cite{Slob58}.

If the boundary $\partial D$ of the domain $D$ is sufficiently smooth, 
then, using the standard volume form $d\sigma$ on the hypersurface $\partial D$ induced 
from ${\mathbb R}^n$, 
we may consider the Sobolev-Slobodeckii spaces  $H^s(\partial D)$ on 
$\partial D$. Namely, let 
$L^2(\partial D)$ be the Lebesgue space of 
functions on  $\partial D$ with the inner product
\begin{equation*}
\left( u,v \right)_{L^2(\partial D)} = \int_{\partial D} v (x)u (x)\ d\sigma (x) .
\end{equation*}
If $0<s<1$ and $\partial D\in C^1$ then we define $H^s(\partial D)$
to be the completion of $C^{1} (\partial D)$ with respect to the norm
$$
   \| u \|_{H^s (\partial D)}
 = \Big( \| u \|^2_{L^{2} (\partial D)}
       + \int \!\!\! \int_{\partial D \times \partial D}
         \frac{|u (x) - u (y)|^2}{|x - y|^{n-1+2s}}\,
         d\sigma (x) d\sigma (y)
   \Big)^{1/2}.
$$
For $s\geq 1$ we have to consider more smooth hypersurfaces. 
For instance, if $\partial D \in C^{[s]+1}$) then we may define the space $H^s (\partial D)$
using local coordinates on $\partial D$ and a suitable partition of unity. 

It is known that the functions of $H^s (D)$, where $s > 1/2$, possess well-defined traces on 
the Lipschitz surface $\partial D$. For $s\in \mathbb N$, the trace operator 
\begin{equation} \label{eq.trace}
t_s: H^s (D) \to H^{s-1/2} (\partial D)
\end{equation}
obtained in this way acts continuously, if $\partial D\in C^s$;
moreover, in this case it possesses a bounded right inverse, see for instance 
\cite[Ch.~1, \S~8]{LiMa72}. 

Let us consider now a mathematical model describing the electrical activity of the heart. 
Assume that the myocardial domain $\Omega_m$ is surrounded by a volume conductor 
$\Omega_b $. The total domain, including the myocardium and the human torso 
$\Omega = \Omega_b \cup \overline{\Omega}_m$,  where $\overline \Omega_m$ is the closure of 
heart domain, is surrounded by a non-conductive medium (air). Assuming that intracellular, 
extracellular and extracardiac media are homogeneous and isotropic, denote by  
$M_i$, $M_e$, $M_b$ the conductivity matrices 
in the intra-, extracellular and extracardiac spaces, and  by $\nu_i$, $\nu_e$  the 
outward normal 
vectors to the surfaces of the heart and body volume ($\Omega_m$ and $\Omega_b$), respectively.

Denote by $\nabla$ the gradient operator  and by 
${\rm div}$ the divergence operator in ${\mathbb R}^n$. 
It is convenient to set 
$$
\Delta_e = -{\rm div} \, M_e \nabla , \, \Delta_i = -{\rm div} \, M_i \nabla, \, 
\Delta_b = -{\rm div} \, M_b \nabla .
$$

Assuming that 
$M_i$, $M_e$, $M_b$ are symmetric non-degenerate   $(n\times n)$-matrices 
with real entries, satisfying 
\begin{equation} \label{eq.M.pos}
\zeta \cdot  M \zeta =  \zeta^T M \zeta  >0 \mbox{ for each } \zeta \in 
{\mathbb R}^n\setminus \{0\},
\end{equation} 
we obtain strongly elliptic operators  $\Delta_e$, $\Delta_i$, $\Delta_b$ 
with constant coefficients. 

If the functions $u_i, u_e$ over $\Omega_m$, and the function $u_b$ over $\Omega_b $ 
stand for intra-, extracellular and extracardiac (electrical) 
potentials, respectively, then the  intra-, extracellular and  extracardiac 
(electrical) currents are given by
$$
J_i = -M_i \nabla u_i, \, J_e = -M_e \nabla u_e, J_b = -M_b \nabla u_b, 
$$
respectively. As the intracellular 
(electrical) charge $q_i$ and the extracellular charge $q_e$ should be 
balanced in the heart tissue, using the divergence operator, we arrive at the following 
equations involving the time variable $t$: 
\begin{equation} \label{eq.balance.charge}
\frac{\partial (q_i+q_e)}{ \partial t} =0 \mbox{ in } \Omega_m,
\end{equation}
\begin{equation} \label{eq.balance.current.i}
-\Delta_i u _i = \frac{\partial q_i}{ \partial t} +\chi I_{\rm ion} \mbox{ in } \Omega_m,
\end{equation}
\begin{equation} \label{eq.balance.current.e}
-\Delta_e u _e = \frac{\partial q_e}{ \partial t} -\chi I_{\rm  ion} \mbox{ in } \Omega_m,
\end{equation}
where $I_{\rm  ion}$ is the ionic current across the cell membrane and 
$\chi I_{\rm  ion}$ is  ionic current per unit tissue. 
Of course, the charge densities $q_i$, $q_e$ and 
the potentials $u_e$, $u_i$ are actually defined 
on different domains: extracellular and intracellular spaces, respectively.
Thus, equations \eqref{eq.balance.charge}, \eqref{eq.balance.current.i} and 
\eqref{eq.balance.current.e}  reflect the fact that a homogenization procedure 
is at the bottom of the considered model. 

Next, combining  \eqref{eq.balance.charge}, \eqref{eq.balance.current.e}
\eqref{eq.balance.current.i} we obtain the conservation law for the total 
current $(J_i+ J_e)$:
$$
\Delta_i u_i + \Delta_e  u_e  = 0 \mbox{ in } \Omega_m.
$$
In the heart surrounded by a conductor, the normal component of the total current 
should be continuous across the boundary of the heart:
\begin{equation}\label{eq.balance.boundary}
\nu_i \cdot (J_i +J_e) = \nu_i \cdot J_b \mbox{ on } \partial \Omega_m.
\end{equation}
Taking in account the current behaviour at the torso,  
 we arrive at a steady-state version of the bidomain model of the 
electrocardiography \cite{geselowitz1983}, \cite{1}, \cite{2}:
\begin{align}
\Delta_i u_i + \Delta_e  u_e  = 0 \mbox{ in } \Omega_m, \label{eq7} \\
\Delta_b  u_b  = 0\mbox{ in } \Omega_b ,\label{eq8} \\
	u_e  = u_b   \mbox{ on } \partial \Omega_m , \label{eq9} \\
	\nu_i \cdot (M_e \nabla u_e ) 
	= - \nu_e \cdot (M_b \nabla u_b)\mbox{ on }  \partial \Omega_m,  \label{eq10} \\
	\nu_i \cdot (M_i \nabla u_i)  =0 \mbox{ on } \partial \Omega_m,  \label{eq11} \\
	\nu_e \cdot (M_b \nabla u_b)  = 0 \mbox{ on } \partial \Omega,  \label{eq12}
\end{align}
where \eqref{eq10}, \eqref{eq11} are consequences of \eqref{eq.balance.boundary} and 
the assumption that the intracellular domain is completely insulated.

However,  the primary equations 
\eqref{eq.balance.charge}, \eqref{eq.balance.current.i}, 
\eqref{eq.balance.current.e}  are actually evolutionary. 
That is why the model contains a large evolutionary part, too. 
For example, \eqref{eq.balance.current.i}, 
\eqref{eq.balance.current.e} imply
\begin{equation} \label{eq.balance.cable.1}
-\Delta_i u _i + \Delta_e u _e = \frac{\partial (q_i-q_e)}{ \partial t} +2\chi I_{\rm  ion} 
\mbox{ in } \Omega_m \times (0,T),
\end{equation}
On the other hand, the so-called transmembrane potenitial 
$$
v=u_i-u_e 
$$
 satisfies \begin{equation} \label{eq.balance.cable.2}
u_i-u_e = \frac{1}{2} \frac{q_i-q_e}{\chi \, C_m} \mbox{ in } \Omega_m
\end{equation}
where $C_m$ is the capacitance of the cell membrane. 
Thus, using \eqref{eq.balance.cable.1} and \eqref{eq.balance.cable.2} we arrive 
at the so-called cable equation
\begin{equation} \label{eq.balance.cable.0}
\frac{1}{2\chi}\Big(-\Delta_i u _i + \Delta_e u _e\Big) = 
 C_m \frac{\partial (u_i-u_e)}{ \partial t} + I_{\rm  ion} 
\mbox{ in } \Omega_m \times (0,T),
\end{equation}
see, for instance, \cite[\S 2.2.2]{2}. Of course, there 
are many possibilities to supplement the model by more 
advanced and complicated relations. 
 
We begin the discussion from the 
 following problem that 
is known as the steady inverse problem of the electrocardiography. 

\begin{problem} \label{pr.inverse.inside}
 Given 
 the values of electrical potential $u_b$ on the boundary of the body 
\begin{equation} 
	u_b = f_0 \mbox{ on }  \partial \Omega,
	\label{eq6}
\end{equation}
find the intracellular potential $u_i$ and extracellular potential $u_e$ in 
$\Omega_m$ and extracardiac potential $u_b$ in $ \Omega_b$, 
 satisfying equations \eqref{eq7}, \eqref{eq8} and 
boundary conditions \eqref{eq9}-\eqref{eq12}. 
\end{problem}

Note that Problem \ref{pr.inverse.inside} includes the Cauchy problem \eqref{eq8}, 
\eqref{eq12}, \eqref{eq6} for the elliptic operator $\Delta_b$ that is 
usually ill-posed in all the standard function spaces, see, for instance, 
\cite{Lv1}, \cite{Tark36},  or elsewhere. The uniqueness 
theorem for the Cauchy problem   
related to elliptic equations (see, for example, \cite[Theorem 2.8]{ShTaLMS}) 
provides the uniqueness of the potential $u_b$ in the Lebesgue and the Sobolev type spaces  
 if it exists. However, the uniqueness of the potentials $u_i$,  $u_e$, 
satisfying relations \eqref{eq7}-\eqref{eq12}, \eqref{eq6} 
and the transmembrane potential $v $ was not mathematically established. 

The uniqueness of solutions to Problem \ref{pr.inverse.inside} 
was investigated in \cite{2K} in 
Hardy type spaces over smooth domains: 
$$
{\mathcal H} (\Omega) = \{ u\in H^1 (\Omega), \, \frac{\partial u}{\partial \nu}
\in L^2 (\partial\Omega), \, 
\Delta u\in L^2 (\Omega) \}
$$
where $\Delta = \sum_{j=1}^n \frac{\partial^2}{\partial x^2_j}$ 
is the usual Laplace operator in ${\mathbb R}^n$ and 
$\frac{\partial}{\partial \nu}$ is the normal derivative 
with respect to $\partial \Omega$.  
However it was essential in \cite{2K}  that the matrices $M_i$, $M_e$, 
were proportional:
\begin{equation*}
M_e = \gamma M_i 
\end{equation*}
with some positive number $\gamma$. In particular, this means that 
a linear change of variables reduces the consideration to the situation 
where 
\begin{equation} \label{eq.scalar}
\Delta_i = -\sigma_i \Delta, \, \Delta_e = -\sigma_e \Delta,  
\, \gamma= \frac{\sigma_e}{\sigma_i},
\end{equation}
and $\sigma_i $, $\sigma_e $, 
are positive numbers characterizing the 
 electrical conductivity of the corresponding media. 
It was proved that under this very restrictive assumption the null-space of 
Problem \ref{pr.inverse.inside} consists of all the triples 
\begin{equation} \label{eq.nullspace.inside.prop}
	\left\{ 
	\begin{array}{ccccc}
		 u_b  & = & 0 & \rm{in} & \Omega_b,\\
	 u_e  & = & u & \rm{in} & \Omega_m,\\
	 u_i  & = & -\frac{\sigma_e}{\sigma_i}  u + c& \rm{in} & \Omega_m,\\
\end{array}
\right.
\end{equation}
 where $c$ is an arbitrary constant and $u $ is an arbitrary function 
from $ {\mathcal H} (\Omega_m)$ satisfying 
$$
u= \frac{\partial u}{\partial \nu} = 0 \mbox{ on } \partial \Omega_m,
$$
cf. Theorem \ref{t.bidomain.null.gen} below. 
In this way the so-called transmembrane potential $v$
can be defined on $\partial \Omega_m$ up to an arbitrary constant summand $c$; 
it can be uniquely defined on $\partial \Omega_m$
if we  supplement the bidomain model with the following 
calibration assumption that always is achievable for isotropic conductivity:
there is a constant $c_0$ such that 
\begin{equation} \label{eq.calibration}
\int_{\partial \Omega_m} (u_i + c_0 u_e)(y) d\sigma (y) =0.
\end{equation}

As for the Existence Theorem for Problem \ref{pr.inverse.inside},
 the ill-posed Cauchy problem \eqref{eq8}, \eqref{eq12}, \eqref{eq6} can be treated 
with the standard regularization methods described in \cite{KMF91}, \cite{Lv1},  
 \cite{Tark36}. Moreover for the potentials $u_i, u_e$ in this 
very particular case  we have 
\begin{equation} \label{eq.sol.ui.prop}
u_i =-\frac{\sigma_e}{\sigma_i} u_e +\frac{\sigma_e}{\sigma_i} {\mathcal N}_i (0, \nu_i \cdot (M_b \nabla u _b)) +c
\end{equation}
where 
$c$ is an arbitrary constant, $u_e $ is an arbitrary function 
from $ {\mathcal H} (\Omega_m)$ satisfying \eqref{eq9}, \eqref{eq10} and 
${\mathcal N}_i (g, u_0)$ is the unique solution to the Neumann problem
\begin{equation} \label{eq.Neumann.M}
\left\{ \begin{array}{lll}
\Delta_i {\mathcal N}_i (g, u_0) =g & \rm{in} & \Omega_m,\\
\nu_i \cdot (M_i \nabla u ) {\mathcal N}_i (g, u_0) = u_0  & \rm{on} & \partial \Omega_m,\\
\end{array}
\right.
\end{equation}
satisfying 
\begin{equation} \label{eq.Neumann.normilize}
\int_{\partial D} {\mathcal N}_i (g, u_0)(x) d\sigma(x) =0.
\end{equation}
Of course, Problem \ref{eq.Neumann.M} is not always solvable but it 
has the Fredholm property in many Sobolev type 
spaces. More precisely, it is solvable in the standard Sobolev spaces
 if and only if 
\begin{equation} \label{eq.Green.M.Neumann}
\int_{\partial \Omega_m}  u_0 d\sigma + \int_{\Omega_m} g dx =0,
\end{equation}
see, for instance, \cite{Simanca1987}. The last identity can be easily 
verified for the data chosen in \eqref{eq.sol.ui.prop} because 
of the relations in the bidomain model, see also Theorem \ref{t.bidomain.exists.gen} below 
for a more general situation.  
Again, if calibration assumption \eqref{eq.calibration} holds for the pair $u_i$, $u_e$  
then the constant $c$ in \eqref{eq.sol.ui.prop} may be uniquely defined by 
\begin{equation}  \label{eq.const}
c= -c_0\Big(\int_{\partial \Omega_m}  d\sigma (y) \Big)^{-1}
\int_{\partial \Omega_m}  u_b (y) d\sigma (y) . 
\end{equation}
In particular, this scheme gives a possibility to find the potential $u_b$ in $\Omega_b$ 
and the potentials $u_i, u_e, v$ on $\partial \Omega_m$.

But, from mathematical point of view, this means that in the present form the 
steady part of the bidomain model \eqref{eq7}--\eqref{eq12}, \eqref{eq6}  has too many 
degrees of freedom. It seems, that one  equation related to the  
potentials $u_i$, $u_e$  in $\Omega_m$ is still missing.

In the next sections we will discuss what kind of equation 
can be added to steady bidomain model even in a much more general situation in order 
to provide uniqueness of its solutions. We will also discuss the 
uniqueness of solutions to an  evolutionary bidomain model.

\section{A more general steady problem}
\label{s.bidomain.gen}

Let us consider a more general steady problem. 
With this purpose, recall that a linear (matrix) differential operator 
$$
A (x,\partial) = \sum_{|\alpha|\leq m} A_\alpha (x) \partial^\alpha
$$  
 of order $m$  with $(l\times k)$-matrices $A_\alpha (x) $, having 
entries from $C^\infty (X)$ on an open set $X \subset {\mathbb R}^n$, 
is called an operator  with injective symbol on $X$ 
if $l\geq k$ and for its principal symbol
$$
\sigma(A) (x,\zeta) = \sum_{|\alpha|= m} A_\alpha (x) \zeta^\alpha
$$ 
we have 
$$
\mathrm{rang} \, ( \sigma(A) (x,\zeta) )=k \mbox{ for any } x\in X , \zeta \in {\mathbb R}^n 
\setminus \{0\}.
$$
An operator $A$  is called (Petrovsky) elliptic, if $l=k$ and 
its symbol is injective (or, the same, non-degenerate) on $X$. 

Then let $S_{A}(D)$ be the space of generalized solutions to the equation 
$Ah = 0$ in a domain $D$. If the operator $A$ has an injective symbol and its 
coefficients are real analytic, then 
the Petrovsky theorem yields that the elements of the space
$S_{A}(D)$ are real analytic vector functions in $D$.

An operator $L (x,\partial)$ is called strongly elliptic if 
it is elliptic, its order is even (and equals to  $2m$) 
and there is a positive constant $c_0$ such that 
$$
(-1)^{m}\Re{\, (w^*\sigma(L)} (x,\zeta) \,w)\geq c_0 |\zeta|^{2m} |w|^2
\mbox{ for any } x\in X , \zeta \in {\mathbb R}^n , 
w \in {\mathbb C}^k
$$
where $w^* = \overline w^T$ and $w ^T$ is the transposed vector 
for $w \in {\mathbb C}^k$. 

Actually, any operator $A^*A$ is strongly elliptic of order $2m$ if the principal symbol 
of $A$ is injective and  
$$
A^* (x,\partial) = \sum_{|\alpha|\leq m} (-1)^{|\alpha|} \partial^\alpha (A^*_\alpha (x) \cdot)
$$
is the formal adjoint for $A$ with the adjoint matrices $A^*_\alpha (x) $. 
The typical operator of such type are the (minus) Laplacians 
$$
-\Delta = \nabla^*\nabla = - {\rm div} \, \nabla, \, 
\Delta_M = -{\rm div} \, M \nabla,
$$
where $M$ is a self-adjoint non-negative non-degenerate $(n\times)$-matrix 
with constant entries. 

Next, we recall that a set of linear differential 
operators $\{B_0,B_1, \dots B_{m-1}\}$ is called a $(k\times k)$-matrix 
Dirichlet system of order $(m-1)$ on $\partial D$ if 

1) the operators are defined in a neighbourhood of $\partial D$; 

2) the order of the differential operator $B_j $ equals to $j$;

3) the map $ \sigma (B_j) (x,\nu (x)) :{\mathbb C}^k \to {\mathbb C}^k$ 
is bijective for each $x \in \partial D$, where 
$\nu (x)$ will denote the outward normal vector to the hypersurface $\partial D$
at the point $x\in \partial D$, see \cite{Roit96}, \cite[\S 9.2.2]{Tark36}.

According to the Trace Theorem, if $\partial D\in C^{s}$, $s\geq m\geq 1$ then 
each operator $B_j$ induces a bounded linear operator
\begin{equation} \label{eq.trace.B_j}
B_j: [H^s (D)]^k \to [H^{s-j-1/2} (\partial D)]^k.
\end{equation}
Easily, if a first order operator $A$ has injective symbol in a neighbourhood 
of the closure $\overline D$ of a smooth domain $D$ then 
the pair $\{I_k, \sigma ^* (A)(\nu) A  \}$ is a Dirichlet system of the second order 
near $\partial D$ where $I_k$ is the unit $(k\times k)$-matrix.

Now we may proceed with the formulation of the transmission problem. 

Let $\Omega_m$ and  
$\Omega$ be smooth bounded domains in ${\mathbb R}^n$, $n\geq 2$,  such that 
$\Omega \supset \overline \Omega_m$ where $\overline \Omega_m$ is the closure of the  
domain $\Omega_m$. We set $\Omega_b = \Omega \setminus \overline \Omega_m$. 
Denote by  
$A^{(i)}$, $A^{(e)}$, $A^{(b)}$ matrix differential operators 
with \mbox{real analytic} 
coefficients and  injective symbols  in some neighbourhoods $U_m$ and $U_b$  of 
the compacts $\overline{\Omega}_m$ and $\overline{\Omega}_b$, respectively. 
Then the differential operators  
$$
\Delta^{(e)} = (A^{(e)})^* A^{(e)} , \, \Delta^{(b)} = (A^{(i)})^* A^{(i)}, \, 
\Delta^{(b)} = (A^{(b)})^* A^{(b)}
$$
are elliptic and strongly elliptic over $U_m$ and $U_b$, respectively. 

We also fix the first order $(k\times k)$-matrix boundary operators 
 $B^{(b)}_1$ near $\partial \Omega_b$ and 
$B^{(i)}_1$, $B^{(e)}_1$ near $\partial \Omega_m$ such that 
$(I_k, B^{(b)}_1)$, $(I_k, B^{(i)}_1)$ and $(I_k, B^{(e)}_1)$ 
are Dirichlet pairs and 
\begin{equation} \label{eq.Bi}
\int_{\partial \Omega_m} v^* B^{(i)}_1 u d \sigma = 
\int_{\Omega_m} \Big( (A^{(i)}v)^* A^{(i)}u -v^* \Delta ^{(i)} u\Big)dy , 
\end{equation}
\begin{equation} \label{eq.Be}
\int_{\partial \Omega_m} v^* B^{(e)}_1 u d \sigma = 
\int_{\Omega_m} \Big( (A^{(e)}v)^* A^{(e)}u -v^* \Delta ^{(e)} u \Big)dy 
\end{equation}
for all $u \in [H^2 (\Omega_m)]^k$, $v \in [H^1 (\Omega_m)]^k$, 
\begin{equation} \label{eq.Bb}
\int_{\partial \Omega_b} v^* B^{(b)}_1 u d \sigma = 
\int_{\Omega_b} \Big( (A^{(b)}v)^* A^{(b)}u-v^* \Delta ^{(b)} u  \Big)dy 
\end{equation}
for all $u \in [H^2 (\Omega_b)]^k$, $v \in [H^1 (\Omega_b)]^k$. For instance, one may take 
\begin{equation}\label{eq.boundary.op}
B^{(b)}_1 = \sigma ^* (A^{(b)})(\nu_e)  A^{(b)}, \, 
B^{(i)}_1 = \sigma ^* (A^{(i)})(\nu_i)  A^{(i)}, \, 
B^{(e)}_1 =\sigma ^* (A^{(e)})(\nu_i)  A^{(e)},
\end{equation}
where $\nu_i$ and $\nu_e$  are the outward normal 
vectors to the surfaces $\Omega_m$ and $\Omega_b$, 
respectively, because by Ostrogradsky-Gauss formula we have:
\begin{equation} \label{eq.Green.M.one}
\int_{\partial D} v^* \sigma^* (A)(\nu) A  u d\sigma = 
\int_{D} \Big(  (A v)^*  A u   - v ^* A^*A u \Big) dx
\end{equation}
for all  $v \in [H^1 (D)]^k$, $u\in  [H^2 (D)]^k$ and 
any  first order differential operator $A$ with an injective symbol over $\overline D$. 

\begin{problem} \label{pr.inverse.inside.gen}
Let $s\geq 2$ and $\alpha_i, \alpha_e, \beta_e,  \beta_i  
\in \mathbb R$, $\alpha_e^2 +\alpha_i ^2  \ne 0$, $\beta_e^2 +\beta_i ^2\ne 0$. 
Given vector functions 
$$
f\in [H^{s-2} (\partial \Omega)]^k, \,  
f_0\in [H^{s-1/2} (\partial \Omega)]^k, \,  
f_1\in [H^{s-3/2} (\partial \Omega)]^k, 
$$
find,
if possible,  vector functions  $u_i, u_e \in [H^s (\Omega_m)]^k$, 
$u_b \in [H^s (\Omega_b)]^k$  satisfying 
\begin{align}
\alpha_i \Delta^{(i)} u_i + \alpha_e  \Delta^{(e)}  u_e  = 0 \mbox{ in } \Omega_m, 
\label{eq7.gen} \\
\Delta^{(b)}  u_b  = f\mbox{ in } \Omega_b ,\label{eq8.gen} \\
	u_e  = u_b   \mbox{ on } \partial \Omega_m , \label{eq9.gen} \\
			B^{(e)}_1 u_e  
	= 		\beta_e B^{(b)}_1	u_b \mbox{ on }  \partial \Omega_m , \label{eq10.gen} \\
	B^{(i)}_1 u_i
	=  \beta_i 
	B^{(b)}_1 u_b 
	\mbox{ on } \partial \Omega_m,  \label{eq11.gen} \\
	B^{(b)}_1 u_b  = f_1   
	\mbox{ on } \partial \Omega,  \label{eq12.gen} \\
	u_b = f_0 \mbox{ on }  \partial \Omega.
	\label{eq6.gen}
\end{align}
\end{problem}

Besides we will use below an assumption that is similar to calibration 
relation \eqref{eq.calibration}:
there is a constant $c_0$ such that 
\begin{equation} \label{eq.calibration.gen}
\int_{\partial \Omega_m} h^* (y) (u_i + c_0 u_e)(y) d\sigma (y) =0 
\mbox{ for all } h \in S_{A^{(i)} } (\Omega_m) \cap [H^s (\Omega_m)]^k.
\end{equation}

Problem \ref{pr.inverse.inside.gen} was considered in \cite{Shef} in 
the Sobolev spaces over smooth domains  
under the following rather restrictive but partially natural assumptions:

\begin{enumerate}
\item[1)]  the coefficients of the operators $A^{(i)}$, $A^{(e)}$, $A^{(b)}$ 
are constants;
\item[2)] there is a constant $\gamma>0$ such that
$A^{(e)} = \sqrt{\gamma} A^{(i)}$;
\item[3)]
the spaces  
$S_{A ^{(e)} }(\Omega_m)\cap [H^2 (\Omega_m)]^k$ and 
$S_{A^{(b)}}(\Omega)\cap [H^2 (\Omega)]^k$ 
coincide;
\item[4)] boundary operators 
$B^{(b)}_1$, $B^{(i)}_1$, $B^{(e)}_1$ are given by \eqref{eq.boundary.op};
\item[5)]
$\alpha_i =1$, $\alpha_e =1$, 
$\beta_e=-1$, $\beta_i=0$,  $f_1\equiv 0$, $f\equiv 0$; 
\item[6)]
the Shapiro-Lopatinsky conditions are 
fulfilled for the pair of differential 
operators $(\Delta ^{(i)}, B^{(i)}_1 )$ over 
$\overline{\Omega}_m$, see \cite[Chapter 1, \S 3, 
condition II for $q=0$]{Agr}, \cite{Shap} or elsewhere. 
\end{enumerate}

Only assumption 6) from the list is essential for our 
considerations because relations \eqref{eq7.gen}, \eqref{eq11.gen} 
lead us to the following Neumann problem:
given  pair $g \in [H^{s-2} (\Omega_m)]^k$ and $u_1 \in [H^{s-3/2} (\partial \Omega_m)]^k$, 
find, if possible,  a function $u \in [H^{s} (\Omega_m)]^k$ such that 
\begin{equation} \label{eq.Neumann.M.gen}
\left\{ \begin{array}{lll}
\Delta^{(i)} u =g & \rm{in} & \Omega_m,\\
	B^{(i)}_1
	u = u_1  & \rm{on} & \partial \Omega_m,\\
\end{array}
\right.
\end{equation}
see, for instance, \cite{Simanca1987}. 
More precisely, the Shapiro-Lopatinsky conditions provide that problem 
\eqref{eq.Neumann.M.gen} has the Fredholm property. 
Practically, under \eqref{eq.Bi}, they are equivalent to the following 
bound: there is a positive constant $c_i$ such that  
$$
\|u\|_{H^1 (\Omega_m)}\leq c _i\|A^{(i)}u \|_{L^2 (\Omega_m)} 
\mbox{ for all } u \in (S_{A^{(i)}}(\Omega_m)\cap [H^1 (\Omega_m)]^k)^\bot
$$
where $(S_{A^{(i)}}(\Omega_m)\cap [H^1 (\Omega_m)]^k)^\bot$ stands for 
the orthogonal complement of the subspace $S_{A^{(i)}}(\Omega_m)\cap [H^1 (\Omega_m)]^k$ 
in the Hilbert space $ [H^1 (\Omega_m)]^k$. 
In particular, the Shapiro-Lopatinsky conditions guarantee that the space 
$S_{A^{(i)}}(\Omega_m)\cap [H^s (\Omega_m)]^k$ is 
finite dimensional. Actually, the following theorem holds true.

\begin{theorem} \label{t.Neumann.M.gen}
Let $s \in \mathbb N$, $s\geq 2$, and $\partial \Omega_m\in C^s$. 
If Shapiro-Lopatinsky conditions holds true 
the pair $(\Delta^{(i)}, B^{(i)}_1  )
$ and 
\eqref{eq.Bi} is fulfilled then 
Neumann problem \eqref{eq.Neumann.M.gen} is solvable if and only if 
\begin{equation} \label{eq.Green.M.Neumann.gen}
\int_{\partial \Omega_m}  h^* u_1 d\sigma = \int_{\Omega_m } h^* g dx 
\mbox{ for all } h \in S_{A^{(i)}}(\Omega_m)\cap [H^s (\Omega_m)]^k.  
\end{equation}
The null-space of problem \eqref{eq.Neumann.M.gen}  coincides 
with $S_{A^{(i)}}(\Omega_m)\cap [H^s (\Omega_m)]^k$. 
Moreover, under \eqref{eq.Green.M.Neumann} there is only one solution $u$ satisfying 
\begin{equation} \label{eq.Neumann.normilize.gen}
\int_{\partial D} h^*(x) u(x) d\sigma(x) =0 
\mbox{ for all } h \in S_{A^{(i)}}(\Omega_m)\cap [H^s (\Omega_m)]^k.
\end{equation}
\end{theorem}

\begin{proof} See, for instance, \cite{Simanca1987}.
\end{proof}

The unique solution to \eqref{eq.Neumann.M.gen}  satisfying 
\eqref{eq.Neumann.normilize.gen} will be denoted by ${\mathcal N}^{(i)} (g,u_1)$. 

Let us discuss the new steady transmission problem, conditions providing the uniqueness 
of its solutions and the ways to weaken exceedingly strong assumptions 1)--5) above. 

First, we note, that similarly to Problem \ref{pr.inverse.inside}, 
Problem \ref{pr.inverse.inside.gen} includes the Cauchy problem \eqref{eq8.gen}, 
\eqref{eq12.gen}, \eqref{eq6.gen} for the elliptic operator $\Delta^{(b)}$ that is 
usually ill-posed in all the standard function spaces, see, for instance, 
\cite{Lv1}, \cite{Tark36},  or elsewhere. However, the uniqueness 
theorem for the Cauchy problem related to elliptic equations (see, for example, 
\cite[Theorems 2.8]{ShTaLMS}) provides the uniqueness of the vector function
 $u_b$ in the Lebesgue and the Sobolev type spaces  if it exists. Moreover, 
since the operator $\Delta^{(b)}$ is elliptic and its coefficients are real analytic, 
it admits a bilateral fundamental solutions to the operator, say  $\varphi_b $, in a 
neighbourhood of the compact $\overline \Omega_b$. 
Let us indicate a solvability criterion and  
formulas for its solutions based on results from  \cite{ShTaLMS}, \cite{FeSh14} and 
\cite{Tark36}. With this purpose, we set 
$$
F(x) = \int_{\Omega_b} \Big(  \varphi _b (x,y)^* f(y) dy + 
$$
$$
\int_{\partial \Omega} 
(B^{(b)}_1 (y,\partial _y) \varphi 
_b (x,y))^* f_0(y) 
-  (\varphi 
_b (x,y))^* f_1(y) \Big) d\sigma (y)  , \,\, x \not \in \partial \Omega. 
$$
\begin{theorem} \label{t.Cauchy.bidomain.gen} 
Let $B^{(b)}_1$ satisfy \eqref{eq.Bb}. Then 
Cauchy problem \eqref{eq8.gen}, \eqref{eq12.gen}, \eqref{eq6.gen} with 
data  in $[H^{s-2} (\Omega_b)]^k \times [H^{s-1/2} (\partial \Omega)]^k\times 
[H^{s-3/2} (\partial \Omega)]^k$ is densely solvable in 
the space $[H^s(\Omega_b)]^k$. Moreover, it has  no more that one solution 
in $[H^s(\Omega_b)]^k$. 
It is solvable for a triple $f \in [H^{s-2} (\partial \Omega)]^k$,  
$f_0 \in [H^{s-1/2} (\partial \Omega)]^k$, 
$f_1 \in [H^{s-3/2} (\partial \Omega)]^k$  if and only if there is 
a function  ${\mathcal F}\in [H^s (X\setminus \overline \Omega_m)]^k$ satisfying 
$\Delta_b {\mathcal F} = 0$ in  $X\setminus \overline \Omega_m$ and such that 
$$
{\mathcal F}(x) = 
F(x) 
$$
for all $ x \in X\setminus \overline \Omega$. 
Besides, the solution $u$, if exists, is given by the following formula
\begin{equation} \label{eq.sol.Cauchy.bidomain.gen}
u _b (x) =F(x)  - {\mathcal F}(x), \,  x \in  \Omega_b.
\end{equation}
\end{theorem}

\begin{proof} Follows from \cite[Theorems 2.8 and 5.2]{ShTaLMS}
 for the case $f=0$ and \cite{FeSh14} for $f\ne 0$ 
because both the set $S=\partial \Omega$ where the boundary Cauchy data are defined and 
its complement $\partial \Omega_b \setminus S = 
\partial \Omega_m$ are non empty and open in  the relative topology.
\end{proof}

The ill-posed Cauchy problem \eqref{eq8.gen}, \eqref{eq12.gen}, \eqref{eq6.gen} 
can be also treated with the standard regularization methods described in \cite{KMF91}, 
\cite{Lv1},  \cite{Tark36}, providing formulas for exact and approximate solutions.

Now we are ready to describe the null-space of 
Problem \ref{pr.inverse.inside.gen}. We slightly differ 
the approach of \cite{2K} and \cite{Shef} in this more general 
situation.

\begin{theorem} \label{t.bidomain.null.gen}
Let $s\geq 2$, \eqref{eq.Bi}, \eqref{eq.Be}, \eqref{eq.Bb} be fulfilled and the  
pair $(\Delta^{(i)}, 	
B^{(i)}_1 )$ satisfy Shapiro-Lopatinsky conditions in $\overline \Omega_m$. If 
$\alpha_i \ne 0$ and 
\begin{equation} \label{eq.kernels.weak}
S_{A^{(i)}} (\Omega_m) \cap [H^s (\Omega_m)]^k \subset 
S_{\Delta^{(e)}} (\Omega_m) \cap [H^s (\Omega_m)]^k 
\end{equation}
then the null-space of Problem \ref{pr.inverse.inside.gen} consists of all the triples 
$u_i$, $u_e$, $u_b$ from $[H^s (\Omega_m)]^k\times [H^s (\Omega_m)]^k \times 
[H^s (\Omega_b)]^k$ satisfying the following conditions: 
\begin{equation} \label{eq.nullspace.inside.gen}
	\left\{ 
	\begin{array}{ccccc}
	u_b  & = & 0 & {\rm in} & \Omega_b,\\ 
	u_e  & = & u & {\rm in} & \Omega_m,\\
	 u_i  & = & (-\alpha_e/\alpha_i)  {\mathcal N}^{(i)} (\Delta^{(e)} u,0) 
	+h_0& {\rm in} & \Omega_m,\\
\end{array}
\right.
\end{equation}
where $h_0$ is an arbitrary element of the finite dimensional space  
 $S_{A^{(i)}} (\Omega_m) \cap [H^s (\Omega_m)]^k$ and $u $ 
is an arbitrary vector function 
from $ [H^2_0 (\Omega_m) \cap H^s (\Omega_m)]^k$. Moreover, if 
calibration assumption \eqref{eq.calibration.gen} holds for a pair $u_i$, $u_e$ from 
the null-space then the element $h_0$ in \eqref{eq.nullspace.inside.gen} equals to zero.  
\end{theorem}

\begin{proof} Indeed, let the triple $(u_b, u_i,u_e) \in 
[H^s (\Omega_b)]^k \times [H^s (\Omega_m)]^k \times [H^s (\Omega_m)]^k  $ belong to 
the null-space of Problem \ref{pr.inverse.inside.gen}. Hence $f \equiv 0$ in $\Omega_b$, 
$f_0=f_1\equiv 0$ on $\partial \Omega$ and then 
$u_b \equiv 0$ in $\Omega_b$ because of  Theorem \ref{t.Cauchy.bidomain.gen}. 
Of course, using \eqref{eq9.gen}, \eqref{eq10.gen}, we obtain 
$$
u _e =u_b =  B^{(e)}_1  u_e =  \beta_e B^{(b)}_1 u_b  =0 
\mbox{ on } \partial \Omega _m 
$$ 
for $u_e \in [H^s (\Omega_m)]^k$. Since the pair $(1,  B^{(e)}_1)$ 
is a Dirichlet system on $\partial \Omega_m$, then, according to \cite{HedbWolf1}, 
$u_e \in [H^2_0 (\Omega_m) \cap H^s (\Omega_m)]^k$. 
The function $u_i$ satisfies \eqref{eq7.gen} and \eqref{eq11.gen} and 
hence 
$$
B^{(i)}_1 u_i = \beta_i 
 B^{(b)}_1 u_b =0 
$$
Then, according to Theorem \ref{t.Neumann.M.gen}, 
this means precisely 
$$
u_i =(-\alpha_e/\alpha_i) {\mathcal N}^{(i)} (\Delta ^{(e)} u_e,0) +h_0
$$
with an arbitrary element $h_0\in S_{A^{(i)}} (\Omega_m) \cap [H^s (\Omega_m)]^k$. 

Thus, any triple  $(u_b, u_i,u_e) \in 
H^2 (\Omega_b) \times H^2 (\Omega_m) \times H^2 (\Omega_m)  $, belonging to 
the null-space of Problem \ref{pr.inverse.inside.gen}, has the form 
as in \eqref{eq.nullspace.inside.gen} with 
an arbitrary element $h_0$ of the finite dimensional space  
 $S_{A^{(i)}} (\Omega_m) \cap [H^s (\Omega_m)]^k$ and   an arbitrary 
vector function $ u=u_e\in [H^s (\Omega_m) \cap H^2_0 (\Omega_m)]^k$. 

Let a triple  $(u_b, u_i,u_e) \in 
[H^s (\Omega_b)]^k \times [H^s (\Omega_m)]^k \times [H^s (\Omega_m)]^k  $ have the form 
as in \eqref{eq.nullspace.inside.gen} with an 
arbitrary element $h_0\in S_{A^{(i)}} (\Omega_m) \cap 
[H^s (\Omega_m)]^k$  and an arbitrary vector 
function $ u\in [H^s (\Omega_m) \cap H^2_0 (\Omega_m)]^k$. 
Then, obviously,  $f\equiv 0$ in $\Omega$, $f_0=f_1\equiv 0$ on $\partial \Omega$. 
Moreover, integrating by parts with the use of \eqref{eq.Be}
we easily obtain
$$
\int_{\Omega_m}h^* \Delta^{(e)} u dy = \int_{\Omega_m} (\Delta^{(e)} h )^* u dy
=0 \mbox{ for all } h \in S_{A^{(i)}} (\Omega_m) \cap 
[H^s (\Omega_m)]^k
$$
because $u \in H^2 _0 (\Omega_m)$ and embedding \eqref{eq.kernels.weak} 
is fulfilled. Hence,  as \eqref{eq.Bi} holds true, Theorem \ref{t.Neumann.M.gen} implies that 
 there is a solution $w$ to Neumann problem \eqref{eq.Neumann.M.gen} for 
the operator $\Delta^{(i)}$:
\begin{equation*} 
\left\{ \begin{array}{lll}
\Delta^{(i)} w =(-\alpha_e/\alpha_i) \, \Delta^{(e)} u & \rm{in} & \Omega_m,\\
	B^{(i)}_1 w= 0  & \rm{on} & \partial \Omega_m.\\
\end{array}
\right.
\end{equation*}
According to Theorem \ref{t.Neumann.M.gen}, 
the general form of such a solution is precisely 
$$
w= (-\alpha_e/\alpha_i)  {\mathcal N}^{(i)} (\Delta_e u,0) +h_0 
$$
with 
an arbitrary element $h_0\in S_{A^{(i)}} (\Omega_m) \cap [H^s (\Omega_m)]^k$. 
If we take $u_i = w$ then 
\begin{equation} \label{eq.nullspace.inside.1.gen}
	\left\{ 
	\begin{array}{ccccc}
	 u_b  & = & 0 & \rm{in} & \Omega_b,\\
	 u_e  & = & u & \rm{in} & \Omega_m,\\
	B^{(e)}_1 	u_e
	& = & 0 & \rm{on} & \partial \Omega_m,\\ 
	u_e  & = & 0 & \rm{on} & \partial \Omega_m,\\
		\Delta^{(i)} u_i  & = & (-\alpha_e/\alpha_i)  \Delta^{(e)} u & \rm{in} & \Omega_m,\\
	B^{(i)}_1 u_i	& = & 0 & \rm{on} & \partial \Omega_m\\ 
\end{array}
\right.
\end{equation}
with any $u\in [H^2_0 (\Omega_m) \cap H^s (\Omega_m)]^k$. 

Thus, any triple  $(u_b, u_i,u_e) \in 
[H^s (\Omega_b)]^k \times [H^2 (\Omega_m)]^k \times [H^2 (\Omega_m)]^k  $ having the form 
as in \eqref{eq.nullspace.inside.gen} with a arbitrary element $h_0 
\in S_{A^{(i)}} (\Omega_m) \cap [H^s (\Omega_m)]^k$ and an arbitrary vector 
function $ u\in [H^2_0 (\Omega_m) \cap H^s (\Omega_m)]^k$ belongs to the 
null-space of Problem \ref{pr.inverse.inside.gen}. 

Finally, if calibration assumption \eqref{eq.calibration.gen} is fulfilled then, 
as $u_e =u \in [H^2 _0(\Omega_m)]^k$, condition \eqref{eq.Neumann.normilize.gen} yields 
$$
0=\int_{\partial \Omega_m} h^*(u_i + c_0 u_e)(y) d\sigma (y) 
= \int_{\partial \Omega_m} h^* h_0   d\sigma 
$$
for all $S_{A_i} (\Omega_m) \cap [H^s (\Omega_m)]^k$. In particular, 
for $h=h_0$ we obtain 
$$
\int_{\partial \Omega_m} |h_0|^2   d\sigma  =0
$$
i.e. $h_0 = 0$ on $\partial \Omega_m$. 
Finally, as $h_0 \in S_{A^{(i)}} (\Omega_m) \cap [H^s (\Omega_m)]^k$ 
we may use  the Uniqueness Theorem \cite[Theorem 2.8]{ShTaLMS}
applying it to the Cauchy problem 
\begin{equation*} 
\left\{ \begin{array}{lll}
A^{(i)} h_0 =0 & \rm{in} & \Omega_m,\\
h_0 =  0  & \rm{on} & \partial \Omega_m,\\
\end{array}
\right.
\end{equation*}
and concluding that $h_0 \equiv 0$ in $\Omega_m$.
\end{proof}

Let us formulate an existence theorem for the Problem 
\ref{pr.inverse.inside.gen}. 
With this purpose we note that the scale of Sobolev spaces can be extended 
for negative smoothness indexes, too. 
Namely, for $s > 0$, denote by $H^{-s} (D)$ the completion of
$C^{\infty} (\overline{D})$ with respect to the norm
$$
   \| u \|_{H^{-s} (D)}
 = \sup_{\substack{v \in C^{\infty}_{\mathrm{comp}} (D) \\ v \ne 0}}
   \frac{|(v,u)_{L^2 (D)}|}{\| v \|_{H^{s} (D)}}.
$$
Actually, $H^{-s} (D)$ can be identified with the dual of    $H^s_0 (D)$
with respect to the pairing induced by $(\cdot, \cdot)_{L^2 (D)}$.

We denote also by $\Pi_0$  the $[L^2 (\partial \Omega_m)]^k$-orthogonal 
projection onto the closed finite dimensional subspace 
$S_{A_i} (\Omega_m) \cap [H^s (\Omega_m)]^k \subset [L^2 (\partial \Omega_m) ]^k$.
Also, we assume that the following relations are fulfilled:
\begin{equation} \label{eq.kernels.strong.A}
S_{A^{(i)}} (\Omega_m) \cap [H^s (\Omega_m)]^k \subset 
S_{A^{(e)}} (\Omega_m) \cap [H^s (\Omega_m)]^k,  
\end{equation}
\begin{equation} \label{eq.kernels.strong.B}
S_{A^{(i)}} (\Omega_m) \cap [H^s (\Omega_m)]^k 
= S_{A^{(i)}} (\Omega) \cap [H^s (\Omega)]^k \subset 
S_{A^{(b)}} (\Omega_b) \cap [H^s (\Omega_b)]^k. 
\end{equation}

\begin{theorem} \label{t.bidomain.exists.gen}
Let $s\geq 2$, \eqref{eq.Bi}, \eqref{eq.Be},  \eqref{eq.Bb} be fulfilled and 
 the  pair $(\Delta^{(i)}, B^{(i)}_1)$  
satisfy 
Shapiro-Lopatinsky conditions in $\overline \Omega_m$  
and let  embeddings \eqref{eq.kernels.strong.A}, \eqref{eq.kernels.strong.B}
hold true. If $\alpha_i \ne 0$ then, given $f \in [H^{s-2} (\partial \Omega_b)]^k$, 
$f_0 \in [H^{s-1/2} (\partial \Omega_b)]^k$, 
$f_1 \in [H^{s-3/2} (\partial \Omega_b)]^k$
admitting 
the solution $u_b\in [H^s (\Omega_b)]^k$  to \eqref{eq8.gen}, \eqref{eq12.gen} and 
\eqref{eq6.gen}, 
there are functions $u_e,u_i\in [H^s (\Omega_m)]^k$
satisfying \eqref{eq7.gen}, \eqref{eq9.gen}, \eqref{eq10.gen}, \eqref{eq11.gen} if and only 
if 
\begin{equation}\label{eq.trans}
(\beta_e \alpha_e + \alpha_i \beta_i  ) 
\Big(\int_{\Omega_b} h^* (y) f(y) dy + 
\int_{\partial \Omega_b} h^* (y) f_1(y) d\sigma (y) \Big)=0
\end{equation}
for all $ h \in S_{A^{(b)}} (\Omega_b) \cap [H^s (\Omega_b)]^k$. 
\end{theorem}

\begin{proof} Indeed, as $u_b\in [H^s (\Omega_b)]^k$  we see that 
$$
u_b \in  [H^{s-1/2} (\partial \Omega_b) ]^k, 
\, B^{(b)}_1 u_b \in  [H^{s-3/2} (\partial \Omega_b) ]^k.
$$
Since $(I_k, B^{(e)} _1) $ is a Dirichlet pair, applying \cite[Lemma 5.1.1]{Roit96} we may 
find a vector function $u_e \in [H^s (\Omega_m)]^k$ satisfying \eqref{eq9.gen}, 
\eqref{eq10.gen}. Of course, such a vector function is not unique and there are several 
ways to construct it.   For example, one may take $u_e$ as 
the unique solution $u\in [H^s (\Omega_m)]^k$ to Dirichlet problem
\begin{equation} \label{eq.Dirichlet.Laplacian.Q.gen}
\left\{ \begin{array}{lll}
Q u =g & \rm{in} & \Omega_m,\\
 u = u_b & \rm{on} & \partial \Omega_m,\\
 B^{(e)}_1 u	=  
	\beta_e B^{(b)}_1	u_b  &  \rm{on} &  \partial \Omega_m   \\
\end{array}
\right.
\end{equation}
with an arbitrary function $g\in [H^{s-4} (\Omega_m)]^k$ and an arbitrary 
strongly elliptic formally non-negative operator $Q$ of the fourth order and with real 
analytic coefficients in a neighbourhood of $\overline \Omega_m$. 
Indeed, under these assumptions 
the Dirichlet problem \eqref{eq.Dirichlet.Laplacian.Q.gen} admits one and only 
one solution in $[H^{s} (\Omega_m)]^k$, 
 for instance, \cite{Mikh}, \cite[Ch. 5]{Roit96} or elsewhere.
For instance, one may take $Q=(\Delta^{(e)}) ^2$ because
 $\Delta^{(e)}=(\Delta^{(e)})^*$ and hence the operator 
$$
(\Delta^{(e)})^2 = (\Delta^{(e)})^* \Delta^{(e)}
$$ 
is strongly elliptic formally non-negative and of fourth order; in particular, 
the linear space  $S_Q (\Omega_m )\cap [H^{2}_0 (\Omega_m)]^k$ is trivial. 

If $\alpha_e\ne 0$ then, integrating by parts with the use of 
\eqref{eq.Be}, \eqref{eq.Bb}, \eqref{eq8.gen}, \eqref{eq10.gen}, \eqref{eq12.gen}, 
\eqref{eq.kernels.strong.A}, \eqref{eq.kernels.strong.B}, we obtain
\begin{equation} \label{eq.chain.existence}
-\frac{\alpha_e}{\alpha_i} 
\int_{\Omega_m} h^*(y)\Delta^{(e)} u_e (y) dy = 
\frac{\alpha_e}{\alpha_i} \int_{\partial \Omega_m}  h^* B^{(e)}_1  u_e d\sigma  
-\frac{\alpha_e}{\alpha_i} \int_{\Omega_m} (A^{(e)} h )^* A^{(e)} u_e (y) dy 
=
\end{equation}

$$ 
(\beta_e +\frac{\beta_i-\beta_i}\alpha_e/\alpha_i ) 
\frac{\alpha_e}{\alpha_i} \int_{\partial \Omega_m}  h^* B^{(b)}_1 u_b d\sigma 
-(\beta_e \frac{\alpha_e}{\alpha_i} +\beta_i )\int_{\partial \Omega} h^*\Big( 
B^{(b)} _1u _b - B^{(b)} _1u _b \Big)d\sigma 
=
$$
$$
-(\beta_e \alpha_e/\alpha_i +\beta_i) \Big(\int_{\Omega_b} (A^{(b)}h)^* A^{(b)} u_b 
 dy 
- \int_{\Omega_b} h^* \Delta^{(b)} u_b  
dy\Big)   
$$
$$
-\beta_i \int_{\partial \Omega _m} h^*
B^{(b)} _1u _b  d\sigma + 
(\beta _e \alpha_e/\alpha_i +\beta_i)\int_{\partial \Omega} h^* 
B^{(b)} _1u _b d\sigma =
$$
$$
-\beta_i \int_{\partial \Omega _m} h^*
B^{(b)} _1u _b  d\sigma +  \frac{\alpha_e\beta_e   + \alpha_i\beta_i}{\alpha_i}
\Big(\int_{\Omega_b} h^* f  dy  + \int_{\partial \Omega} h^*
f_1  d\sigma \Big) 
$$

for all $h \in S_{A^{(i)}} (\Omega_m) \cap [H^s (\Omega_m)]^k$. 

If $\alpha_e= 0$ then, similarly, 
\begin{equation} \label{eq.chain.existence.2}
\beta_i 
\int_{\partial \Omega_m} h^*
B^{(b)} _1u _b  d\sigma = -\beta_i \int_{\partial \Omega_b} h^*
B^{(b)} _1u _b  d\sigma + 
\beta_i \int_{\partial \Omega} h^*
B^{(b)} _1u _b  d\sigma = 
\end{equation}
\begin{equation*}
\beta_i \Big(\int_{\partial \Omega} h^*
f_1  d\sigma +  \int_{\Omega_b} h^* f  dy \Big)
\end{equation*}
for all $h \in S_{A^{(i)}} (\Omega_m) \cap [H^s (\Omega_m)]^k$. 

In any case, \eqref{eq.trans}, \eqref{eq.chain.existence}, 
\eqref{eq.chain.existence.2}  and 
Theorem \ref{t.Neumann.M.gen} yield the existence of a vector function 
$u_i \in [H^s (\Omega_m)]^k$ satisfying 
\begin{equation} \label{eq.Neumann.M.i.gen}
\left\{ \begin{array}{lll}
\Delta^{(i)} u_i = (-\alpha_e/\alpha_i)  \Delta^{(e)} u_e & {\rm in} & \Omega_m,\\
B^{(i)}_1u_i = \beta_i B^{(b)} _1u _b  & \rm{ on } & \partial \Omega_m.\\
\end{array}
\right.
\end{equation}
More precisely, Theorem \ref{t.Neumann.M.gen} states that $u_i$ is given by 
\begin{equation} \label{eq.sol.ui.gen}
u_i = {\mathcal N}^{(i)} ((-\alpha_e/\alpha_i) \Delta^{(e)} u_e,\beta_i 
B^{(b)}_1 u_b) +h_0
\end{equation}
with an arbitrary element $h_0 \in S_{A^{(i)}} (\Omega_m) \cap [H^s (\Omega_m)]^k$. 

Again, if calibration assumption \eqref{eq.calibration.gen} holds for a pair $u_i$, $u_e$ 
 then 
$$
0=\int_{\partial \Omega_m} h^*\Big( h_0+  {\mathcal N}^{(i)} ((-\alpha_e/\alpha_i)  
\Delta_e u_e,\beta_i B^{(b)}_1 u_b)  +c_0 u_e\Big) d\sigma (y) =
$$
$$
\int_{\partial \Omega_m} (\Pi_0 h)^*  \Big( h_0  +c_0 u_e\Big) d\sigma =
\int_{\partial \Omega_m} h^*  \Big( h_0  +c_0 \Pi_0 u_b\Big) d\sigma 
$$
for any $h \in S_{A^{(i)}} (\Omega_m) \cap [H^s (\Omega_m)]^k$ 
because of normalising condition \eqref{eq.Neumann.normilize}. 
Thus, the element $h_0$ in \eqref{eq.sol.ui.gen} may be uniquely defined by 
\begin{equation}  \label{eq.const.gen}
h_0  = -c_0 \Pi_0 u_b .
\end{equation}

Finally, chain of equalities \eqref{eq.chain.existence}
tell us that condition \eqref{eq.trans} is necessary 
for the solvability of problem \eqref{eq7.gen}, \eqref{eq9.gen}, \eqref{eq10.gen}, 
\eqref{eq11.gen}, that was to be proved. 
\end{proof}

As we have mentioned above, in the very particular case where, additionally, 
the operators $A^{(e)}$, $A^{(i)}$, $A^{(b)}$ are proportional, 
similar theorems were proved in \cite{Shef}.

\begin{remark} 
Assumptions \eqref{eq.kernels.weak}, \eqref{eq.kernels.strong.A}, \eqref{eq.kernels.strong.B}  
are fulfilled in the following reasonable situation:
\begin{equation} \label{eq.holonom}
A^{(i)} = M^{(i)} {\mathcal A}, \, A^{(e)} =  M^{(e)} {\mathcal A}, \, A^{(b)} = M^{(b)} {\mathcal A}, 
\end{equation}
where $M^{(i)}$,  $M^{(e)}$,  $M^{(b)}$ are $(l\times l)$-matrices with real analytic 
coefficients over $U_m$ and $U_b$ satisfying \eqref{eq.M.pos} and 
${\mathcal A}$ is an $(l\times k)$ {\it holonomic}  differential 
operator, i.e. it is an operator with constant coefficients such that for any domain $D \subset {\mathbb R}^n$ we have
$$
S_{\mathcal A} (D) = S_{\mathcal A} ({\mathbb R}^n) , \, \mathrm{dim} (S_{\mathcal A} 
({\mathbb R}^n)  <\infty,
$$
and 
there is a positive constant $c_{\mathcal A}$ such that  
$$
\|u\|_{H^1 (\Omega)}\leq c \|{\mathcal A} u \|_{L^2 (\Omega)} 
\mbox{ for all } u \in (S_{{\mathcal A}}(\Omega)\cap [H^1 (\Omega)]^k)^\bot
$$
for any bounded domain $\Omega$ with smooth boundary. 
 
For instance, in the models of the electrocardiography we have 
${\mathcal A} = \nabla$, 
$$
S_{\mathcal A} (D) = 
S_\nabla (D) = S_\nabla ({\mathbb R}^n) = S_{\mathcal A} ({\mathbb R}^n)  ={\mathbb R}, \, 
\mathrm{dim} (S_{\mathcal A} ({\mathbb R}^n)  =1, 
$$
for any domain $D\subset {\mathbb R}^n$
and then
$$
S_{A^{(i)}} (\Omega_m) = S_{A^{(i)}} (\Omega)
= S_{A^{(e)}} (\Omega_m) = S_{A^{(b)}} (\Omega_b) ={\mathbb R},
$$
$$
S_{A^{(i)}} (\Omega_m) \subset S_{\Delta^{(e)}} (\Omega_m),
$$
i.e. \eqref{eq.kernels.weak}, \eqref{eq.kernels.strong.A}, \eqref{eq.kernels.strong.B} are 
fulfilled in this case, too.
\end{remark}

Of course, in some particular situations we can 
say much more. 

\begin{example} \label{ex.prop.gen}
Consider the situation where  the case $A^{(e)}$, $A^{(i)}$ and   $A^{(b)}$ 
satisfy \eqref{eq.holonom}  and the operators $A^{(e)}$, $A^{(i)}$ are proportional  
(cf. \cite{2K}, \cite{Shef}). Then, in particular, 
the  pair $(\Delta^{(i)},  \sigma ^* (A^{(i)})(\nu_i)  A^{(i)})$ 
satisfies  
Shapiro-Lopatinsky conditions in $\overline \Omega_m$, 
 embeddings \eqref{eq.kernels.weak}, \eqref{eq.kernels.strong.A}, \eqref{eq.kernels.strong.B}
hold true and 
\begin{equation} \label{eq.prop.gen}
\Delta^{(e)} = \gamma \Delta^{(i)} 
\mbox{ with some } \gamma >0. 
\end{equation}

Then, with any  pair $(\Delta^{(i)}, B^{(i)}_1)$ 
satisfying Shapiro-Lopatinsky conditions and \eqref{eq.Bi}, we have 
$$
{\mathcal N}^{(i)} (\Delta^{(e)} u,0) = \gamma {\mathcal N}^{(i)} (\Delta^{(i)} 
u,0) =  \gamma u
$$
for each $u \in [H^s (\Omega_m) \cap H^2_0 (\Omega_m)]^k$. Thus,  
according to Theorem \ref{t.Neumann.M.gen}, 
\begin{equation} \label{eq.nullspace.inside.prop.gen}
	\left\{ 
	\begin{array}{ccccc}
	 u_e  & = & u & {\rm in} & \Omega_m,\\
	 u_i  & = & (-\alpha_e/\alpha_i) \gamma u + h_0 & {\rm in} & \Omega_m,\\
\end{array}
\right.
\end{equation}
for each pair $u_i$, $u_e$ from 
the null-space of Problem \ref{pr.inverse.inside} where 
$h_0$ is an arbitrary element of $S_{A^{(i)}} 
(\Omega_m) \cap [H^s (\Omega_m)]^k$ and $u $ is an 
arbitrary function 
from $ [H^s (\Omega_m) \cap H^2_0 (\Omega_m)]^k$.  Again, if 
calibration assumption \eqref{eq.calibration.gen} holds for a pair $u_i$, $u_e$ from 
the null-space then the element $h_0$ in \eqref{eq.nullspace.inside.prop.gen} equals to zero.

As for the Existence Theorem, in this case 
$$
{\mathcal N}^{(i)} ((-\alpha_e/\alpha_i) \Delta^{(e)} u_e,\beta_i B^{(b)}_1 u_b) =
$$
$$ 
 {\mathcal N}^{(i)} ((-\alpha_e/\alpha_i) \gamma\Delta^{(i)} u_e, 
\Big((-\alpha_e/\alpha_i)\gamma \beta _e+ 
\beta_i- (-\alpha_e/\alpha_i)\gamma\beta_e)  B^{(b)}_1 u_b)
= 
$$
$$
-(\alpha_e/\alpha_i) \gamma u_e +(\beta_i + 
(\alpha_e/\alpha_i)\gamma \beta_e) 
{\mathcal N}^{(i)} (0, B^{(b)}_1 u_b).
$$
Thus, formula \eqref{eq.sol.ui.gen} has the form
\begin{equation} \label{eq.sol.ui.prop.gen}
u_i =-(\alpha_e/\alpha_i) \gamma u_e +(\beta_i + 
(\alpha_e/\alpha_i) \gamma \beta_e) {\mathcal N}^{(i)} (0, B^{(b)} u_b) +h_0
\end{equation}
where $h_0$ is an arbitrary element of $S_{A^{(i)}} (\Omega_m) \cap [H^s (\Omega_m)]^k$, 
if \eqref{eq.trans} is fulfilled. 
Again, if calibration assumption \eqref{eq.calibration.gen} holds for the pair $u_i$, $u_e$  
then the element $h_0$ in the last formula may be uniquely defined by \eqref{eq.const.gen}.
\end{example}

\begin{remark} 
\label{r.alphai}
We note that for $\alpha_i\ne 0$, 
similarly to the results of \S \ref{s.bidomain.steady},  
Theorem \ref{t.bidomain.null.gen}  means 
that Problem \ref{pr.inverse.inside.gen} has too many degrees of freedom. 
Actually, for $\alpha_i = 0$ the problem is even more unbalanced. Indeed, 
using  \eqref{eq8.gen}, \eqref{eq9.gen}, \eqref{eq10.gen} and the Uniqueness Theorem 
\cite[Theorem 2.8]{ShTaLMS} for the Cauchy problem, it is easy to check 
that in this case 
the null-space of Problem \ref{pr.inverse.inside.gen} consists of all the triples $u_i$, $u_e$, 
$u_b$ from $[H^s (\Omega_m)]^k\times [H^s (\Omega_m)]^k \times [H^s (\Omega_b)]^k$ satisfying the following conditions: 
\begin{equation} \label{eq.nullspace.inside.gen.alphai}
	\left\{ 
	\begin{array}{ccccc}
	u_b  & = & 0 & {\rm in} & \Omega_b,\\ 
	u_e  & = & 0 & {\rm in} & \Omega_m,\\
	 u_i  & = & u & {\rm in} & \Omega_m,\\
\end{array}
\right.
\end{equation}
where $u \in [H^s (\Omega_m)]^k$ is an arbitrary 
function satisfying 
$$
B^{(i)}_1 u= 0   \mbox{ on }  \partial \Omega_m.
$$
On the other hand, as $\alpha_e \ne 0$, 
the existence of a solution to 
Problem \ref{pr.inverse.inside.gen} depends upon the solvability of the 
following Cauchy problem  
\begin{equation} \label{eq.exist.gen.alphai}
	\left\{ 
	\begin{array}{ccccc}
	\Delta^{(e)}u_e  & = & 0 & {\rm in} & \Omega_m,\\ 
	u_e  & = & u_b & {\rm on} & \partial \Omega_m,\\
	 u_e  & = & \beta _e  B^{(e)}_1 u_b & {\rm on} & \partial \Omega_m,\\
\end{array}
\right.
\end{equation}
corresponding to \eqref{eq8.gen}, \eqref{eq9.gen}, \eqref{eq10.gen}. 
Since the boundary data are given on all the surface $\partial \Omega_m$, 
this Cauchy problem is normally solvable in the declared spaces, see  
\cite[Theorems 2.8 and 5.2]{ShTaLMS}. However it 
has a large co-kernel: it is solvable if and only if 
$$
\int_{\partial \Omega_m} 
(B^{(e)}_1 (y,\partial _y) \varphi 
_e (x,y))^* u_b(y) 
-  \beta_e (\varphi 
_e (x,y))^*B^{(b)}_1 u_b (y) \Big) d\sigma (y) =0 
$$
for all $ x \in \Omega$. In particular, this implies that the data 
$f_0 \in [H^{s-1/2} (\partial \Omega)]^k$, 
$f_1 \in [H^{s-3/2} (\partial \Omega)]^k$, $f 
\in  [H^{s-2} (\Omega_b)]^k$, derfining the function $u_b$, 
can not be arbitrary. For this reason, we will concentrate 
our efforts on the case where $\alpha_i\ne 0$. 
\end{remark}

As we have noted in Remark \ref{r.alphai},  
the null-space of Problem \ref{pr.inverse.inside.gen} is too large.  
Practically, this means that 
 at least one  equation related to the unknown 
vector functions  in $\Omega_m$ is still missing. On the other hand, the proof of  
 Theorem \ref{t.bidomain.exists.gen} suggests us to supplement 
Problem \ref{pr.inverse.inside.gen} with 
a fourth order strongly elliptic equation
\begin{equation} \label{eq.add}
Q u_e =g \mbox{ in } \Omega_m 
\end{equation}
with a given function $g$ in $\Omega_m$.

\begin{corollary} \label{c.bidomain.exists.gen}
Let $s\geq 2$, $\alpha_i \ne 0$, \eqref{eq.Bi}, \eqref{eq.Be}, \eqref{eq.Bb} be fulfilled and 
the  pair $(\Delta^{(i)},  B^{(i)}_1)$ satisfy 
Shapiro-Lopatinsky conditions in $\overline \Omega_m$,  
let  embeddings \eqref{eq.kernels.strong.A}, \eqref{eq.kernels.strong.B}
hold true and the triple 
$$f \in [H^{s-2} (\Omega_b)]^k, \, 
f_0 \in [H^{s-1/2} (\partial \Omega_b)]^k, \,  
f_1 \in [H^{s-3/2} (\partial \Omega_b)]^k
$$  
admit
the solution $u_b\in [H^s (\Omega_b)]^k$  to \eqref{eq8.gen}, \eqref{eq12.gen}, 
\eqref{eq6.gen} and satisfy \eqref{eq.trans}. 
If $Q$ is a fourth order strongly elliptic operator over $\overline \Omega_m$, then,
given vector $g \in [H^{s-4} (\Omega_m)]^k$, problem   
\eqref{eq7.gen}, \eqref{eq9.gen}, \eqref{eq10.gen}, \eqref{eq11.gen}, \eqref{eq.add} has the  
Fredholm property. If 
$Q$ is a fourth order formally non-negative strongly elliptic operator 
with real analytic coefficients over $\overline \Omega_m$, then,
given vector $g \in [H^{s-4} (\Omega_m)]^k$, problem   
\eqref{eq7.gen}, \eqref{eq9.gen}, \eqref{eq10.gen}, \eqref{eq11.gen}, 
\eqref{eq.calibration.gen}, \eqref{eq.add} has 
one and only one solution $(u_i,u_e) \in [H^s (\Omega_m)]^k \times [H^s (\Omega_m)]^k $.
\end{corollary}

\begin{proof} Recall that a problem related to an operator 
equation 
$$
R u =f
$$
with a linear bounded operator $R: X_1 \to X_2$ in Banach spaces $X_1, X_2$ has the Fredholm 
property, if the kernel ${\rm ker}(R)$ of the operator $R$ and 
the cokernel ${\rm coker}(R)$ (i.e. 
the kernel ${\rm ker}(R^*)$ of its adjoint operator $R^*: X_2^* \to X_1^*$)
are finite-dimensional vector spaces and the range of the operator $R$ is closed in $X_2$.   

Under the hypothesis of this corollary both Dirichlet problem 
 \eqref{eq9.gen}, \eqref{eq10.gen}, \eqref{eq.add}, see, for instance, 
\cite{Roit96}  and 
Neumann problem \eqref{eq7.gen}, \eqref{eq11.gen}, see, 
for instance, \cite{Simanca1987}, have Fredholm property 
in the relevant Sobolev spaces. Hence the first part of the 
statement of the corollary is proved. 

If we additionally assume that  
$Q$ is a fourth order formally non-negative strongly elliptic operator 
with real analytic coefficients over $\overline \Omega_m$, then,
given vector $g \in [H^{s-4} (\Omega_m)]^k$, problem   
Dirichlet problem 
 \eqref{eq9.gen}, \eqref{eq10.gen}, \eqref{eq.add} has one and only 
one solution $u_e \in  [H^{s} (\Omega_m)]^k$.  Moreover, as we have seen in the proof 
of Theorem \ref{t.bidomain.exists.gen}, under calibration condition 
\eqref{eq.calibration.gen}, 
Neumann problem \eqref{eq7.gen}, \eqref{eq11.gen} 
is uniquely solvable in the space $[H^{s} (\Omega_m)]^k$, too. Thus, problem 
\eqref{eq7.gen}, \eqref{eq9.gen}, \eqref{eq10.gen}, \eqref{eq11.gen}, 
\eqref{eq.calibration.gen}, \eqref{eq.add} has 
one and only one solution $(u_i,u_e) \in [H^{s} (\Omega_m)]^k
\times  [H^{s} (\Omega_m)]^k$.

This finishes the proof of the corollary.
\end{proof}

We emphasize that the Fredholm property for a problem is not always the desirable result 
in applications because of the possible lack of the uniqueness and possible absence of 
solutions. As the index (the difference between the 
dimensions of its  kernel and co-kernel) of the Dirichlet problem in the standard setting 
equals to zero, the lack of uniqueness immediately implies some 
necessary solvability conditions applied to the given vector $g$ in the Corollary 
\ref{c.bidomain.exists.gen}.

\begin{example} \label{ex.cardio}
First of all, we note that Problem \ref{pr.inverse.inside} is perfectly fit 
for the new more general model with $\alpha_i =1$, $\alpha_e =1$, 
$\beta_e=-1$, $\beta_i=0$,  $f=0$, $f_1=0$ 
and 
$$
A^{(i)}=M^{(i)} \nabla ,\, A^{(e)}=M^{(e)} \nabla \, A^{(b)}=M^{(b)} \nabla 
$$ 
and $M^{(i)}$,  $M^{(e)}$,  $M^{(b)}$ are $(l\times l)$-matrices with real analytic 
coefficients over $U_m$ and $U_b$ satisfying \eqref{eq.M.pos}.
Comparing with the results of \S \ref{s.bidomain.steady}, we have 
$$
M_i =(M^{(i)} (x))^2, \, M_i =(M^{(i)} (x))^2, \, M_i =(M^{(i)} (x))^2.
$$
i.e. we may consider matrices with real analytic entries. 

Unfortunately, we do not know any published 
modifications of the standard bidomain model of 
the electrocardiography involving higher order strongly elliptic equation
 \eqref{eq.add}. The following example has been reported to us by 
Vitaly Kalinin\footnote{V. Kalinin, MD, PhD, EP Solutions SA, 
Avenue des Sciences 13, 1400 Yverdon-les-Bains, Switzerland, 
e-mail: contact@ep-solutions.ch}.

Namely, consider Problem \ref{pr.inverse.inside} 
in the situation where assumption  \eqref{eq.scalar} is fulfilled. 
Next we assume that the function $f_0$ in \eqref{eq6} does not depend on the time 
variable $t$, 
calibration condition \eqref{eq.calibration} is fulfilled and  
that the following electrodynamic relation holds true for the steady current $u_e$:
\begin{equation} \label{eq.current.2}
\Delta u_e  = - \frac{q_e}{\varepsilon \varepsilon_0}.
\end{equation}

Hence, substituting \eqref{eq.current.2} 
into \eqref{eq.balance.current.i} and \eqref{eq.balance.current.e} 
we obtain formulas that can be useful if we need to transform 
evolutionary equations to stationary ones:
\begin{equation} \label{eq.current.4e}
\frac{\partial \Delta u_e}{\partial t}  =
- \frac{1}{\varepsilon \varepsilon_0} \Big( \sigma _e \Delta u_e +\chi I_{\rm  ion} \Big). 
\end{equation}
Now, taking in account \eqref{eq7}, cable equation \eqref{eq.balance.cable.0} and 
\eqref{eq.sol.ui.prop} 
we obtain the following  equation in the sense of distributions in 
$ \Omega_m \times (0,T)$:
\begin{equation} \label{eq.balance.cable.00.e}
\frac{\sigma_i \sigma_e \varepsilon \varepsilon_0}{\sigma_e+\sigma_i} 
\Delta^2 u _e =  - 
\chi C_m  \sigma_e \Delta u_e  -  C_m  \chi I_{\rm  ion} - 
\frac{\chi \sigma_i \sigma_e \varepsilon \varepsilon_0}{\sigma_e+\sigma_i} 
\Delta I_{\rm  ion}.
\end{equation}
If we are to stay within the framework of linear theory we may 
assume that the ionic current is given by 
\begin{equation} \label{eq.Iion.lin.steady}
I_{\rm  ion} (v) = \sum_{j=1}^n a_j \partial_j v + a_0 v + b  
\end{equation} 
with some function  $b \in L^2 (\Omega)$, and some  constants $a_j$, $0\leq j \leq n$.  
Then, as the operator $\Delta^2$ is strongly elliptic, using \eqref{eq.sol.ui.prop}  and 
\eqref{eq.balance.cable.00.e} we arrive 
at the fourth order strongly elliptic equation
\begin{equation} \label{eq.add.prop}
\frac{\sigma_i \sigma_e \varepsilon \varepsilon_0}{\sigma_e+\sigma_i} 
\Delta^2 u _e + \chi C_m \sigma_e \Delta u_e +
	\Big( \frac{C_m(\sigma_e+\sigma_i)}{\sigma_i } 
\Big) \chi I_{\rm  ion}(u_e)+
\end{equation}
$$
 \chi \sigma_e \varepsilon \varepsilon_0 \Delta I_{\rm  ion}(u_e) 
= - \frac{\sigma_e}{\sigma_i}\chi C_m  I_{\rm  ion} 
({\mathcal N}_i (0, \nu_i \cdot (M_b \nabla u _b))  .
$$
There is little hope that Dirichlet problem  
\eqref{eq.add.prop}, \eqref{eq9}, \eqref{eq10} is uniquely solvable, 
taking in account that  the coefficient $\varepsilon \varepsilon_0 $ is practically 
very small.  Hence we may grant the Fredholm property only 
for problem \eqref{eq7}, \eqref{eq9}, \eqref{eq10}, \eqref{eq11}, \eqref{eq.add.prop} 
even under calibration assumption \eqref{eq.calibration}.

Also we note that in the practical models of the electrocardiography the term 
$I_{\rm  ion} (v,x,t)$ is usually non-linear with respect to $v$. 
For general non-linear Fredholm problems 
one may provide under reasonable assumptions a discrete set of solutions only, 
see \cite{Sm65} for the second order elliptic operators in H\"older spaces.
Thus one should specify the type of the non-linearities under the consideration.
For example, in the models of the cardiology the non-linear term is often taken as 
a polynomial of second or third order with respect to $v$, see, for instance, 
\cite{AP96}, \cite{2}, though these choices do not fully correspond to the 
real processes in the myocardium.  
\end{example}

\begin{example} \label{ex.elastic}
Consider the following $((n^2+1) \times n)$-matrix  differential operator
\begin{equation} \label{eq.A.el}
{\mathscr A} =
\left( \begin{array}{lllll}
\nabla & 0 & 0 &\dots & 0\\
0 & \nabla & 0 & \dots & 0\\
\dots & \dots & \dots & \dots & \dots\\
0 & 0 & 0 &\dots & \nabla\\
\partial_1 & \partial_2 & \partial_3 &\dots & \partial_n\\
\end{array}
\right).
\end{equation}
Its symbol 
$$
\sigma ({\mathscr A})(\zeta) = 
\left( \begin{array}{lllll}
\zeta & 0 & 0 &\dots & 0\\
0 & \zeta & 0 & \dots & 0\\
\dots & \dots & \dots & \dots & \dots\\
0 & 0 & 0 &\dots & \zeta\\
\zeta_1 & \zeta_2 & \zeta_3 &\dots &\zeta_n\\
\end{array}
\right)
$$
is injective  for any $\zeta \in {\mathbb R}^n \setminus \{0\}$ because it contains 
submatrices of the type $\zeta_j I_n$, $1\leq j \leq n$. 

Obviously, the space of its solutions $S_{{\mathscr A}} 
(D)$ coincide with ${\mathbb R}^n$ and the operator 
is holonomic. Taking a diagonal $((n^2+1) \times (n^2+1))$-matrix 
\begin{equation} \label{eq.M.el} 
{\mathscr M} (x) =\left( \begin{array}{lllll}
\mu(x) & 0 & 0 &\dots & 0\\
0 & \mu(x) & 0 & \dots & 0\\
\dots & \dots & \dots & \dots & \dots\\
0 & 0 & 0 & \mu(x) & 0\\
0 & 0  &\dots & 0 & \lambda(x)+ \mu(x)\\
\end{array}
\right)
\end{equation}
with real analytic entries $\lambda(x)$, $\mu(x)$ over a domain $X\subset {\mathbb R}^n$ 
we obtain  $(n\times n)$-differential operator 
$$
{\mathscr A}^* \, {\mathscr M} \, {\mathscr A} = - \Big({\rm div} (\mu (x) I_n) \nabla  +  
\nabla  (\lambda(x)+ \mu(x))  {\rm div}\Big).
$$
In many applications it is known as  the Lam\`e operator; it is elliptic,
 strongly elliptic and formally non-negative if 
$$
\left\{
\begin{array}{lll}
\mu (x) \geq m_0 \mbox{ for all } x \in X, \\
\lambda(x)+ 2\mu(x) \geq m_0 \mbox{ for all } x \in X, \\
\lambda(x)+ \mu(x) \geq 0 \mbox{ for all } x \in X, \\
\end{array}
\right.
$$
where $m_0$ is a positive number, because with $ x\in X$, $\zeta \in {\mathbb R}^n$,
$w \in {\mathbb C}^n$ we have:
$$
\sigma({\mathscr A}^* \, {\mathscr M} \, {\mathscr A}) (\zeta) =
 -  \Big( \mu (x) |\zeta|^2 I_n + (\lambda(x)+ \mu(x)) 
\zeta \zeta ^T \Big), 
$$
$$
\det \sigma({\mathscr A}^* \, {\mathscr M} \, {\mathscr A}) 
(\zeta) =|\zeta|^{2n} \mu ^{n-1}(x) (\lambda(x)+ 2\mu(x)), 
$$
$$
-\Re \Big(w^*\sigma({\mathscr A}^*\, {\mathscr M} 
\, {\mathscr A}) (\zeta)w) \Big) =  \Big( \mu (x) |\zeta|^2 |w|^2 + 
(\lambda(x)+ \mu(x))  
| \zeta ^T w|^2\Big) . 
$$

If the functions $\mu$ and $\lambda$ are constant then 
its bilateral fundamental solution of
convolution type is given by the 
{\it Kelvin-Somigliana matrix}  
 $\Phi (x)  = \left( \Phi_{mj}
(x) \right)$ with the components
$$
\Phi_{mj} (x) =
   \frac{1}{2 \mu (\lambda + 2 \mu)}
   \left( \delta_{mj}\, (\lambda + 3 \mu )  \varphi _n (x)  -
          (\lambda + \mu)\, x_{j}\, \frac{\partial}{\partial x_{m}} \varphi _n (x)
   \right)
	$$
where $\delta_{mj}$ is the Kronecker delta, and $\varphi _n(x)$
is the  standard fundamental solution to the Laplace operator in ${\mathbb R}^n$ 
(see, for example, \cite[Part II, \S 2, (1.7)]{Kup}).

As it is known from the linear Elasticity Theory (for $n=2$ and $n=3$), 
the system of equations
$$
{\mathscr A}^* \, {\mathscr M} \, {\mathscr A}  u = f \mbox{ in } D
$$
describes the displacement vector $u(x)$ of points $x$ of an elastic body $D$ 
under the action of the force $f(x)$; in these case $\mu $ and 
$\lambda   $ are the so-called Lam\'e constants characterizing elastic 
properties of body's material, see, for instance, 
\cite[Ch. 1, \S 11, formula (11.7)]{Kup}.

Next, the matrix ${\mathscr T} = ({\mathscr T}_{mj} (x))$ with the entries 
$$
{\mathscr T}_{mj} (x) =
     \mu\, \delta_{mj}\, \frac{\partial}{\partial \nu } +
   \lambda\, \nu _{m} (x)\, \frac{\partial}{\partial x_{j}} +
       \mu\, \nu _{j} (x)\, \frac{\partial}{\partial x_{m}}
   \, \, \, (m,j = 1,..., n), 
	$$
	is known as the boundary 
	stress operator near $\partial D$ if  $\nu _{i} (x)$
are the components of the outward unit normal vector to $\partial D$
at the point $x$. Applying Ostrogradsky-Gauss formula we see that 
\begin{equation} \label{eq.stress}
\int_{D} v^* {\mathscr T} u d\sigma = 
\int_{D} \Big(  ({\mathscr A} v)^* {\mathscr M} {\mathscr A} u   - v ^* {\mathscr A}^* 
{\mathscr M} {\mathscr A} u \Big) dx
\end{equation}
for all  $v \in [H^1 (D)]^n$, $u\in  [H^2 (D)]^n$, i.e. we may consider 
Problem \ref{pr.inverse.inside.gen} for operators 
$$
A^{(i)} = M^{(i)} {\mathscr A}, \, A^{(e)} = M^{(e)} {\mathscr A}, \, 
A ^{(b)} =  M^{(b)}  {\mathscr A}
$$ 
related to operator \eqref{eq.A.el} and square roots 
 $M^{(i)} $, $M^{(e)}$, $M^{(b)}$ of  
$(n\times n)$-matrices ${\mathscr M}^{(i)}$, 
${\mathscr M}^{(e)}$, ${\mathscr M}^{(b)}$ given by \eqref{eq.M.el} with 
Lam\'e constants $\mu^{(i)}$, $\lambda^{(i)}$, 
$\mu^{(e)}$, $\lambda^{(e)}$, $\mu^{(b)}$, $\lambda^{(b)}$, respectively, 
and boundary first order operators 
$$
B^{(i)}_1={\mathscr T}^{(i)} , \, B^{(e)}_1={\mathscr T}^{(e)} , \,  
B^{(b)}_1={\mathscr T}^{(b)} .
$$   
The problem for an elastic  composite body $\Omega=\Omega_b \cup \overline{ \Omega}_m$
then consists in the following:

1) the description of 
the displacement vector $u_b$ of the `exterior' elastic body $\Omega_b$
by the known force $f$ in $\Omega_b$, the displacement $f_0=u_b$ and  
the stress ${\mathfrak T}^{(b)}u_b=f_1$ on the surface $\partial \Omega$, see \eqref{eq8.gen}, 
\eqref{eq12.gen}, \eqref{eq6.gen};

2) the  description of the displacement vectors $u_i, u_e$ of the `interior' 
composite body $\Omega_m \subset \Omega$, where two more elastic materials are mixed in such 
a way that 

a) the displacement vectors $u_i, u_e$ inside 
 $\Omega_m$ are linked via some homogenization procedure  
 with the use of equation \eqref{eq7.gen};

b) relations  \eqref{eq9.gen}, \eqref{eq10.gen}, \eqref{eq11.gen} connect the 
displacement $u_b$ and the stress ${\mathfrak T}^{(b)}u_b$  with the displacement $u_e$, 
the stress ${\mathfrak T}^{(e)}u_e$ and the stress ${\mathfrak T}^{(i)}u_i$ 
on the surface $\partial \Omega_m$. 

The situation become sufficiently realistic if we assume $\alpha_i =\alpha_e$, $\beta_i=0$, 
$\beta_e=1$,  i.e. the loads applied 
to different materials inside $\Omega_m$ are the same, the stress 
 ${\mathfrak T}^{(i)}u_i$ equals to zero on the surface $\partial \Omega_m$ (for instance, 
because of the corresponding material never contact with the surface) and 
${\mathfrak T}^{(e)}u_e  =-{\mathfrak T}^{(b)}u_b$ on $\partial \Omega_m$.
\end{example}

\section{An evolutionary problem}
\label{s.bidiomain.t}

We recall that the primary equations 
\eqref{eq.balance.charge}, \eqref{eq.balance.current.i}, 
\eqref{eq.balance.current.e}, leading to the classical steady 
bidomain model of the electrocardiography are actually evolutionary. 
That is why, let us obtain a uniqueness theorem for a 
generalized evolutionary problem, too.
With this purpose, we introduce the suitable spaces 
for investigation of parabolic equations,  see, for instance, \cite[Ch. 1]{LadSoUr67}. 

For $T>0$ we set $\Omega_T = \Omega \times (0,T)$.
Let
 $C^{2s,s} (\Omega_T)$ be the set of all continuous functions
 $u$ on  $\Omega_T$, having on $\Omega_T$ continuous partial 
derivatives 
$\partial^j_t \partial^{\alpha}_x u$ for all multi-indexes $(\alpha,j) 
\in {\mathbb Z}_+^{n} \times {\mathbb Z}_+$, satisfying $|\alpha|+2j \leq 2s$. 
Clearly, for the cylinder domain $\Omega_T$ we have 
$\overline{\Omega_T} = \overline{\Omega} \times [0,T]$. 
Then $C^{2s,s} (\overline{\Omega_T} )$ denotes the subset 
in $C^{2s,s} (\Omega_T)$, such that for any function 
$u \in C^{2s,s} 
(\overline{\Omega_T})$ 
and any multi-index  $(\alpha,j) 
\in {\mathbb Z}_+^{n} \times {\mathbb Z}_+$, there is a function $u_{\alpha,j}$, 
continuous on $\overline{\Omega_T}$ and such that 
$\partial^j_t \partial^{\alpha}_x u = u_{\alpha,j}$ in $\Omega_T$. 

Let us denote by  $H^{2s,s} (\Omega_T)$, $s \in  {\mathbb Z}_+$, 
anisotropic (parabolic) Sobolev spaces, see, for instance, 
\cite{LadSoUr67}, i.e. the set of such measurable functions 
 $u$ on  $\Omega_T$ that the partial derivatives 
$\partial^j_t \partial^{\alpha}_x u$ 
belong to the Lebesgue space $L^{2} (\Omega_T)$ 
for all multi-indexes $(\alpha,j) 
\in {\mathbb Z}_+^{n} \times {\mathbb Z}_+$ satisfying $|\alpha|+2j \leq 2s$. 
This is a Hilbert space with the inner product
\begin{equation} \label{eq.Hs}
(u,v)_{H^{2s,s} (\Omega_T)} = \sum_{|\alpha|+2j \leq 2s}
\int_{\Omega_T} \partial^j_t \partial^{\alpha}_x v 
(x,t) \, \partial^j_t \partial^{\alpha}_x u (x,t) dx dt.
\end{equation}
We may also define $H^{2s,s} (\Omega_T)$ as the completion of the linear 
space  $C^{2s,s} (\overline{\Omega_T} )$ with respect to the norm   
$\|\cdot \|_{H^{2s,s} (\Omega_T)}$ induced by the inner product
\eqref{eq.Hs}. In particular, if $s=0$ then  $H^{0,0} (\Omega_T) = 
L^{2} (\Omega_T)$. 

We will also use the so-called Bochner spaces 
of functions depending on 
$(x,t)$ over
$\Omega_T$.
Namely, if $\mathcal B$ is a Banach space (possibly, a space 
of functions over $\Omega$) and  
   $p \geq 1$, we denote by  
$L^p ([0,T],{\mathcal B})$ the Banach space of 
measurable maps 
  $u : [0,T] \to {\mathcal B}$
with the norm
$$
   \| u \|_{L^p ([0,T],{\mathcal B})}
 := \| \|  u (\cdot,t) \|_{\mathcal B} \|_{L^p ([0,T])},
$$
see, for instance, \cite[Ch. \S 1.2]{Lion69}. 

Now, taking in account cable equation \eqref{eq.balance.cable.0}, 
we consider a modified Problem \ref{pr.inverse.inside.gen} adding the 
time variable $t\in [0,T]$.

\begin{problem} \label{pr.inverse.inside.t}
Let $\alpha_i, \alpha_e, \beta_e, \beta_i , \mu_i, \mu_e \in \mathbb R$, 
$\alpha_i^2 + \alpha_e^2\ne 0$, $\beta_e^2+\beta_i^2 \ne 0$, 
$\mu_i^2 + \mu_e^2\ne 0$. 
Given vector functions 
$$f \in [L^2 ( \Omega_m \times (0,T))]^k, \, 
f_0 \in [L^2 ([0,T],H^{3/2} (\partial \Omega_m)]^k, \, 
f_1 \in [L^2 ([0,T],H^{1/2} (\partial \Omega_m)]^k,
$$ 
and  mapping $I:  [H^{2,1} (\Omega_m \times (0,T)]^k  
\to  [L^2 (\Omega_m \times (0,T)]^k$,  
find unknown vector functions $ u_b \in[H^{2,1} (\Omega_b \times (0,T)]^k$, 
$ u_i,u_e \in [H^{2,1} (\Omega_m \times (0,T)]^k$ satisfying  
\begin{align}
\alpha_i \Delta^{(i)} u_i +\alpha_e  \Delta^{(e)}  u_e  = 0 \mbox{ in } 
\Omega_m \times [0,T], 
\label{eq7.gen.t} \\
\Delta^{(b)}  u_b  = f\mbox{ in } \Omega_b \times [0,T],\label{eq8.gen.t} \\
	u_e  = u_b   \mbox{ on } \partial \Omega_m \times [0,T], \label{eq9.gen.t} \\
	B^{(e)}_1 	u_e  	=  \beta_e 
	B^{(b)}_1 u_b \mbox{ on }  \partial \Omega_m 
	\times [0,T], \label{eq10.gen.t} \\ 
	B^{(i)}_1 u_i
	=\beta_i B^{(b)}_1 u_b \mbox{ on } \partial \Omega_m \times [0,T],  \label{eq11.gen.t} \\
	B^{(b)}_1 u_b  = f_1  
	\mbox{ on } \partial \Omega \times [0,T],  \label{eq12.gen.t} \\
	u_b = f_0 \mbox{ on }  \partial \Omega \times [0,T],
	\label{eq6.gen.t} \\
			-\mu_i\Delta^{(i)} u _i +\mu_e \Delta^{(e)} u _e = 
 \frac{\partial (u_i-u_e)}{ \partial t} + I (u_i- u_e) 
\mbox{ in }   \Omega \times (0,T).
\label{eq.balance.cable.0.gen}
\end{align}
\end{problem}

As in \S \ref{s.bidomain.steady}, \S \ref{s.bidomain.gen}, it is reasonable to supplement 
the problem with  calibration assumption: there is a function $c_0 (t) \in C[0,T]$ such that 
\begin{equation} \label{eq.calibration.gen.t}
\int_{\partial \Omega_m} h^* (y) (u_i (y,t) + c_0 (t)u_e (y,t)) d\sigma (y) =0 
\end{equation}
for all $h \in S_{A^{(i)} } (\Omega_m) \cap [H^s (\Omega_m)]^k$ and for almost all
 $t\in [0,T]$.

The further developments  depend essentially on the structure 
of the mapping $I$. We continue the discussion with 
the situation considered in Example \ref{ex.prop.gen}.

\begin{theorem} \label{t.bidomain.null.t.prop}
Let $s\geq 1$, $\alpha_i \ne 0$, $\mu_i\alpha_e + \mu_e \alpha_i \ne 0$, 
\eqref{eq.Bi}, \eqref{eq.Be}, \eqref{eq.Bb} hold true and
 the  pair $(\Delta_i, B^{(i)}_1)$ satisfy 
Shapiro-Lopatinsky conditions in $\overline \Omega_m$. Let also the coefficients 
of the operator $\Delta_i$ constants. If 
\eqref{eq.prop.gen} and \eqref{eq.calibration.gen.t} are fulfilled   and 
\begin{equation} \label{eq.Iion.lin}
I (v) = \sum_{j=1}^n a_j  \partial_j v + a_0  v + g  
\end{equation} 
with some function  vector $g \in [L^2 (\Omega_m \times (0,T))]^k$, and some 
 $(k\times k)$-matrices $a_j$, $0\leq j \leq n$,  
then Problem \ref{pr.inverse.inside.t} has no more than one solution 
$(u_i, u_e, u_b)$ in the space 
$[H^{2,1} (\Omega_m \times (0,T))]^k  \times 
[H^{2,1} (\Omega_m \times (0,T))]^k \times 
[H^{2,1} (\Omega_b \times (0,T))]^k$.   
\end{theorem}

\begin{proof} 
Fix vector functions 
$
f \in [L^2 ( \Omega_m \times (0,T))]^k, \, 
f_0 \in L^2 ([0,T],H^{3/2} (\partial \Omega)), \,  
f_1 \in L^2 ([0,T],H^{1/2} (\partial \Omega))
$, 
admitting a solution 
$u_b \in H^{2,1} (\Omega_b \times (0,T))$ to \eqref{eq8.gen.t},  
\eqref{eq12.gen.t}, \eqref{eq6.gen.t}. 
Let $(\hat u_i, \hat u_e, \hat u_b)$ and $(\tilde u_i, \tilde u_e, \tilde u_b)$  
be two solutions to Problem \ref{pr.inverse.inside.t}. Then
the vector 
$(w_i, w_e,w_b) = (\hat u_i, \hat u_e, \hat u_b) -  
(\tilde u_i, \tilde u_e, \tilde u_b) $ 
satisfies
\begin{align}
\alpha_i \Delta^{(i)} w_i +\alpha_e \Delta^{(e)}  w_e  = 0 \mbox{ in } \Omega_m \times [0,T], 
\label{eq7.gen.t.null} \\
\Delta^{(b)}  w_b  = 0\mbox{ in } \Omega_b \times [0,T],\label{eq8.gen.t.null} \\
	w_e  = w_b   \mbox{ on } \partial \Omega_m \times [0,T], \label{eq9.gen.t.null} \\
	 B^{(e)}_1 	w_e  	= 	\beta_e B^{(b)}_1 w_b \mbox{ on }  \partial \Omega_m 
	\times [0,T], \label{eq10.gen.t.null} \\
			B^{(i)}_1w_i	=	\beta_i B^{(b)}_1 w_b \mbox{ on } \partial \Omega_m \times [0,T],
				\label{eq11.gen.t.null} \\
	 B^{(b)}_1  w_b  = 0  	\mbox{ on } \partial \Omega \times [0,T],  \label{eq12.gen.t.null} \\
	w_b = 0 \mbox{ on }  \partial \Omega \times [0,T],
	\label{eq6.gen.t-null} \\
-\mu_i\Delta^{(i)} w _i +\mu_e \Delta^{(e)} w _e = 
 \frac{\partial (w_i-w_e)}{ \partial t} + I (w_i- w_e) 
\mbox{ in }  \Omega \times (0,T).
\label{eq.balance.cable.0.gen.null}
\end{align}
the last equation being satisfied in $\Omega \times (0,T)$.
Then by Theorem \ref{t.bidomain.null.gen} we have
\begin{equation} \label{eq.nullspace.inside.t}
	\left\{ 
	\begin{array}{ccccc}
	w_b (x,t) & = & 0 & {\rm if} & (x,t)\in \Omega_b \times [0,T],\\ 
	w_e (x,t) & = & w & {\rm if} & (x,t)\in \Omega_m \times [0,T],\\
	 w_i  (x,t)& = & {\mathcal N}^{(i)} ( (-\alpha_e/\alpha_i) \Delta^{(e)} w (\cdot, t),0)(x) &
	{\rm if} & 	(x,t)\in \Omega_m \times [0,T],\\
\end{array}
\right.
\end{equation}
where, as before, ${\mathcal N}^{(i)} $ is the Neumann operator related to $\Delta^{(i)}$ 
and $w $ is a function from the space $L^2 ([0,T], [H^2_0 (\Omega_m)]^k)\cap 
[H^{2s,s} (\Omega_m \times (0,T))]^k$ providing that parabolic equation
\eqref{eq.balance.cable.0.gen.null} is fulfilled and calibration assumption 
\eqref{eq.calibration.gen.t} holds true. 

Since $\Delta^{(i)} = \gamma \Delta^{(i)} $ with some $\gamma>0$, 
then, according to \eqref{eq.nullspace.inside.prop.gen} and \eqref{eq.nullspace.inside.t}, 
we have
\begin{equation} \label{eq.nullspace.inside.t.prop}
	\left\{ 
	\begin{array}{ccccc}
	w_b (x,t) & = & 0 & {\rm if} & (x,t)\in \Omega_b \times [0,T],\\ 
	w_e (x,t) & = & w & {\rm if} & (x,t)\in \Omega_m \times [0,T],\\
	 w_i  (x,t)& = & (-\alpha_e/\alpha_i) \gamma w & {\rm if} & 
	(x,t)\in \Omega_m \times [0,T],\\
\end{array}
\right.
\end{equation}
where $w $ is a function from 
$L^2 ([0,T], [H^2_0 (\Omega_m)]^k) \cap [H^{2s,s} (\Omega_m \times (0,T))]^k$ 
satisfying the following reduced version of 
equation 
\eqref{eq.balance.cable.0.gen.null}: 
\begin{equation}
\label{eq.balance.cable.0.null.r}
 (1+(\alpha_e/\alpha_i) \gamma) \frac{\partial w}{ \partial t} 
 +   (\mu_i(\alpha_e/\alpha_i) +\mu_e) \Delta^{(e)} w = 
\Big(I (\hat u_i-\hat u_e) - 
I (\tilde u_i -\tilde u_e ) \Big) 
\mbox{ in } \Omega_m \times (0,T).
\end{equation}
Clearly, 
\begin{equation} \label{eq.v.w}
\hat v - \tilde v= 
(\hat u_i -\hat u_e)-(\tilde u_i -\tilde u_e ) = w_i-w_e = 
- ((\alpha_e/\alpha_i) \gamma+1)w, 
\end{equation}
and then \eqref{eq.Iion.lin}, \eqref{eq7.gen.t.null}, \eqref{eq.balance.cable.0.null.r} imply
the following relation in $\Omega_m \times (0,T)$:
\begin{equation}
\label{eq.balance.cable.0.null.rr}
(\alpha_i+\alpha_e\gamma) \frac{\partial w}{ \partial t} 
 +   (\mu_e \alpha_i+\mu_i\alpha_e )  \Delta^{(e)} w + 
(\alpha_i+\alpha_e\gamma) \Big( \sum_{j=1}^n a_j \partial_j w + a_0 w\Big) 
=0.
\end{equation}

If $(\alpha_i+\alpha_e\gamma)=0$ then $w(\cdot,t)=0$ for each $t \in [0,T]$
because the Uniqueness Theorem for the Cauchy problem for the second 
order elliptic operator $\Delta^{(e)}$, see \cite[Theorem 2.8]{ShTaLMS}), 
for $w\in [H^2_0 (\Omega_m)]^k $ and $(\mu_i\alpha_e - \mu_e \alpha_i)\ne 0$.  

If $(\alpha_i+\alpha_e\gamma)\ne 0$ then, 
similarly to elliptic theory, we may use integral representations 
in parabolic (backward parabolic) theory. Namely, consider 
the following differential operator 
\begin{equation} \label{eq.L}
{\mathcal L}=
 \frac{\partial }{ \partial t}   +   (\mu_e  \alpha_i+\mu_i\alpha_e )
(\alpha_i+\alpha_e\gamma)^{-1}    \Delta^{(e)}  + 
\sum_{j=1}^n a_j \partial_j  + a_0 
\end{equation}
constant coefficients. 
It is parabolic if $ (\mu_e  \alpha_i+\mu_i\alpha_e )
(\alpha_i+\alpha_e\gamma)^{-1}  >0$ and it is backward 
parabolic if $(\mu_e  \alpha_i+\mu_i\alpha_e )
(\alpha_i+\alpha_e\gamma)^{-1}  <0$. 

We proceed with the parabolic 
case because the arguments for the backward 
parabolic are similar. Indeed, under the asuumtions above ${\mathcal L}$ 
 admits a fundamental solution, say, $\Psi_{\mathcal L}$, 
see \cite[\S 1.5, Theorem 2.8]{eid}, \cite{sol}, and 
hence it admits a suitable integral formula. Namely, 
denote  by $S$ a relatively open subset of $\partial\Omega$ 
and set $S_T = S \times (0,T)$ . 
For functions 
$g \in [L^{2} (\Omega_{T})]^k$, $v \in L^2 ([0,T], [H^{1/2} (\partial\Omega)]^k)$, $w \in 
L^2 ([0,T], [H^{3/2} (\partial\Omega)]^k))$, 
$h \in [H^{1/2}(\Omega)]^k$ we introduce the following potentials:  
\begin{equation*} 
I_{\Omega} (h) (x,t)= \int\limits_{\Omega}\Psi_{\mathcal L}(x,y, t)h(y,0) dy,  
$$
$$
G_{\Omega} (f) (x,t)=\int\limits_{0}^t\int\limits_\Omega \Psi_{\mathcal L}(x,y,\ t,\tau)g(y, \tau)
dy d\tau, 
\end{equation*}
\begin{equation*} 
V_{S} (v) (x,t)=\int\limits_{0}^t\int\limits_{S} \tilde B_0 (y)
\Psi_{\mathcal L}(x,y, t,\tau) v(y, \tau)
ds(y)d\tau, 
\end{equation*}
\begin{equation*} 
W_{S} (w) (x,t)= - \int\limits_{0}^t\int\limits_{S} 
\tilde B_1 (y) 
\Psi_{\mathcal L}(x,y, 
t,\tau)w(y, \tau)ds(y) d\tau, 
\end{equation*}
(see, for instance, \cite[Ch. 1, \S 3 and Ch. 5, \S 2]{frid}), where 
$\tilde B = (\tilde B_0, \tilde B_1)$ is the dual Dirichlet pair 
for the elliptic operator 
$$
D= (\mu_e  \alpha_i+\mu_i\alpha_e ) (\alpha_i+\alpha_e\gamma)^{-1}  \Delta^{(e)}  + 
\sum_{j=1}^n a_j \partial_j  + a_0 
$$ 
and the Dirichlet pair $B=(I_k, B^{(e)}_1)$ over $\partial \Omega$, 
i.e.
\begin{equation*} 
\int_{\partial \Omega} \Big( (\tilde B_1v)^*  u + 
(\tilde B_0v)^* B^{(e)}_1 u 
\Big) d\sigma = 
\int_{D} \Big( v^*  D u - (D^*v)^* u \Big) dx.
\end{equation*}
for all $u,v \in [C^\infty (\overline \Omega)]^k$, see, for instance, 
\cite[Lemma 8.3.3]{Tark97}, \cite[Lemma 9.27]{Tark36}.

By the construction, all these potentials are (improper) integral depending on the 
parameters $(x,t)$.

Next, we formulate the so-called Green formula for the  parabolic operator ${\mathcal L}$.

\begin{lemma} \label{l.Green.heat} 
Assume that the coefficients of the operator $\Delta_i$ and $a_j$, $0\leq j\leq n$, are 
constant. Then for all $ T>0$ and all  $u \in [H^{2,1} (\Omega_{T})]^k$
the following formula holds:
\begin{equation} \label{eq.Green.heat}
\left.
\begin{aligned}
u(x, t),\ (x,t)\in \Omega_{T}  \\
0,\ (x, t)\not\in \overline{\Omega_{T}} 
\end{aligned}
\right\} \! = \Big( 
I_{\Omega} (u)   + G_{\Omega} ({\mathcal L}u)  + 
V _{\partial \Omega} \left( 
\partial _{\nu,M} 
u \right)  +  W_{\partial \Omega} (u) \Big) (x,t)  .
\end{equation}
\end{lemma}

\begin{proof}  See, for instance, \cite[Ch. 6, \S 12]{svesh} or  
\cite[Theorem 2.4.8]{Tark95a} even for more general linear operators admitting 
fundamental solutions or parametrices.
\end{proof}

Taking into account Green formula \eqref{eq.Green.heat}, we obtain
\begin{equation} \label{eq.Green.heat.w}
\left.
\begin{aligned}
w(x, t),\ (x,t)\in \Omega_m \times (0,T)  \\
0,\ (x, t)\not\in \overline{\Omega_m} \times [0,T] 
\end{aligned}
\right\} \! =  
I_{\Omega_m} (w)(x,t)     .
\end{equation}
It is well known that 
the fundamental solution  $\Psi_{\mathcal L} (x,t)$ is real analytic 
with respect to the space variable  $x$ for each $t>0$, see 
\cite{eid}, \cite{sol}. 
In particular, this means that the potential $I_{\Omega} (u)(x,t)$ is real 
analytic with respect to $x$ for each $t>0$, too. However, according to 
\eqref{eq.Green.heat.w}, it equals to zero outside $\overline \Omega_T$.
Therefore it is identically zero for each $t>0$ and then $w\equiv 0$ in $\Omega _T$, 
cf. \cite{KuSh}, \cite{PuSh}  for the similar uniqueness theorem related to 
the heat equation and the parabolic Lam\'e type systems. 

Finally, we see that $(w_i, w_e,w_b)=0$ 
because of \eqref{eq.nullspace.inside.t.prop}.
\end{proof}

Again, we note that Problem \ref{pr.inverse.inside} (i.e. the inverse problem 
of the electrocardiography) 
supplemented with cable equation \eqref{eq.balance.cable.1} 
is perfectly corresponds to Problem \ref{pr.inverse.inside.t} 
with $\alpha_i=1$, $\alpha_e=1$, $\beta_e=-1$, $\beta_i=0$, $\mu_i =1$, $\mu_e =1$, 
$f=0$, $f_1=0$ and the specific choice of the Laplacians $\Delta^{(i)}$, $\Delta^{(e)}$, 
$\Delta^{(b)}$ as in  Example  \ref{ex.cardio}.

On the other hand, we see that Problem \ref{pr.inverse.inside.t}  can be easily adopted 
to many models involving diffusion equations. 
Again, in the practical models of such kind the term 
$I (v,x,t)$ is usually non-linear with respect to $v$. Thus, the uniqueness and 
the existence theorems to Problem \ref{pr.inverse.inside.t}  
under assumptions \eqref{eq.prop.gen} and \eqref{eq.calibration.gen.t} are closely 
related to these type of theorems for the following non-standard Cauchy problem 
for a quasilinear parabolic equation:
\begin{equation}
\label{eq.balance.cable.0.prop}
\left\{
\begin{array}{lll} 
{\mathcal L} v = F(v) &
\mbox{ in } & \Omega_m \times (0,T) ,\\
	v  = g_0  & \mbox{ on } & \partial \Omega_m \times [0,T], 
		\\
	B^{(e)}_1v= g_1
	& \mbox{ on }  & \partial \Omega_m  \times [0,T],
\end{array}
\right.
\end{equation}
with some data $g_0, g_1$ and (possibly, non-linear) term $F$. 
Even in the case where $F$ is linear with respect to $u_e$ 
problem \eqref{eq.balance.cable.0.prop} might be ill-posed in some cases,
cf., \cite{KuSh}, \cite{PuSh}. Thus, for both linear and the non-linear case, 
a thorough investigation of \eqref{eq.balance.cable.0.prop} is necessary. 
It looks that the problem can be treated with the use of Cauchy-Kowalevskaya theorem.
But in a matter of facts, the question is much more delicate because 
the Cauchy-Kowalevskaya theorem is related to real analytic solutions 
for real analytic data in a small neighborhood of a real analytic surface.
Even if we assume that the surface $\partial \Omega_m$ and the data  are
real analytic, the structure of the fundamental solutions to parabolic 
equations with constant coefficients makes us admit that 
solutions of such equations are often real analytic 
with respect to the space variables $(x_1, \dots, x_n)$ but unlikely to 
be analytic with respect the time variable $t$.

However, as the primary goals of the paper were uniqueness theorems, it is worth 
to say that the real analyticity of solutions to \eqref{eq.balance.cable.0.prop} 
with respect to the space variables would leave us a good hope 
for a uniqueness theorem for Problem \ref{pr.inverse.inside.t} 
in a non-linear situation, too.

\smallskip

\textit{Acknowledgments\,} 
The authors were supported by the grant of the Foundation for the Advancement of Theoretical 
Physics and Mathematics "BASIS"{}.


\begin{thebibliography}{XXXXXX}

\bibitem{Agr} 
Agranovich, M.S.,  Vishik, M.I.,  
\textit{Elliptic problems with a parameter and parabolic problems of general type}, 
Rus. Math. Surv., \textbf{19}(1964), 3, 53–157.


\bibitem{AP96}
{Aliev, R.R., Panfilov, A.V.}, 
\textit{A simple two-variable model of Cardiac Excitation}, 
 Chaos, Solitons and Fractals, V. 7:3 (1996), 293--301.

\bibitem{Bor10} 
{Borsuk, M.} 
\textit{Transmission problem for Elliptic second-order Equations in non-smooth domains}, 
Birkh\"auser, Berlin, 2010.

\bibitem{1}  
{Burger, M., Mardal, K.A., Nielsen, B.F.},   
\textit{Stability analysis of the inverse transmembrane potential problem in electrocardiography}, Inverse Problems, 26 (2010), 10, 105012 

\bibitem{eid} {Eidel'man, S.D.}, \textit{Parabolic equations}, 
Partial differential equations – 6, Itogi Nauki i Tekhniki. Ser. Sovrem. Probl. Mat. Fund. 
Napr., 63, VINITI, Moscow, 1990, 201--313.

\bibitem{FeSh14} 
{Fedchenko D.P., Shlapunov, A.A.}, 
\textit{On the Cauchy problem for the elliptic complexes in spaces of distributions}.
Complex Variables and Elliptic Equations, 
V. 59, N. 5, 2014, 651-679. 


\bibitem{frid}
{Friedman, A.}, \textit{Partial differential equations of parabolic type}, 
Englewood Cliffs, NJ, Prentice-Hall, Inc., 1964.

\bibitem{geselowitz1983}
{Geselowitz, D.B., Miller, V.}, \textit{A bidomain model for isotropic cardiac muscle},
Ann. Biomed. Eng., 1983; 11 (3-4), 191-206.

\bibitem{GiTru83} 
Gilbarg, D., Trudinger, N., \textit{Elliptic Partial Differential 
Equations of second order}, Berlin, Springer-Verlag, 1983.

\bibitem{HedbWolf1}
{Hedberg, L.I., Wolff, T.H.},
  Thin sets in nonlinear potential theory,
  Ann. Inst. Fourier (Grenoble) {\bf 33} (1983), no.~4, 161--187.

\bibitem{2K} 
{Kalinin, V., Kalinin A., Schulze W.H.W., Potyagaylo D., Shlapunov, A.,}
\textit{On the correctness of the transmembrane potential based inverse problem 
of ECG}, \textit{Computing in Cardiology}, 2017, 1--4.

\bibitem{KMF91} 
{Kozlov, V.A., Maz'ya V.G., and  Fomin, A.V.},
\textit{An iterative method for solving the Cauchy problem for elliptic equations}, 
USSR Computational Mathematics and Mathematical Physics, 1991, 31:1, 
45--52.

\bibitem{Kup}
{Kupradze,  V.D., Gegelia, T. G., Basheleishvili, M. O.,  
Burchuladze, T. V.} 
\textit{Three-dimensional problems of the mathematical theory of elasticity and 
thermoelasticity},  Applied Mathematics and Mechanics, vol. 25, North-Holland, 1979. 

\bibitem{KuSh}
{Kurilenko, I.A., Shlapunov, A.A.}, 
{On Carleman-type Formulas for Solutions to the Heat Equation}, 
Journal of Siberian Federal University, Math. and Phys.,  12:4 (2019), 421--433.

\bibitem{LadSoUr67}
{Ladyzhenskaya, O.A., Solonnikov V.A., Ural'tseva, N.N.}
\textit{Linear and Quasilinear Equations of Parabolic Type},  
Moscow, Nauka,  1967.

\bibitem{Lv1} {Lavrent'ev, M.M.}, 
\textit{On  the  Cauchy  problem  for  linear 
elliptic equations of the second  order}, Dokl. AN SSSR, 112:2 (1957), 195--197.

\bibitem{Lion69}
{Lions, J.-L.}, \textit{Quelques m\'{e}thodes de r\'{e}solution des probl\`{e}mes aux 
limites non lin\'{e}are}, Dunod/Gauthier-Villars, Paris, 1969, 588~pp.

\bibitem{LiMa72}
{Lions, J. L.,  Magenes, E.},
  \textit{Non-Homogeneous Boundary Value Problems und Applications. Vol. 1},
  Springer-Verlag, Berlin et al., 1972.

\bibitem{McL00}
{McLean, W.},
  \textit{Strongly Elliptic Systems and Boundary Integral Equations},
  Cambridge Univ. Press, Cambridge, 2000.

\bibitem{Mikh} 
{Mikhailov, V.P.}, \textit{Partial differential equations}, 
Nauka, Moscow, 1976.  

\bibitem{PuSh} 
{Puzyrev, R.E., Shlapunov, A.A.} 
\textit{ On a mixed problem for the parabolic Lam\'e type operator}. 
J. Inv. Ill-posed Problems, V. 23:6 (2015), 555--570. 

\bibitem{Roit96}
{Roitberg, Ya.},  \textit{Elliptic Boundary Value Problems in Spaces of Distributions}, 
Kluwer Academic Publishers, Dordrecht, NL, 1996.

\bibitem{Shap} 
{Shapiro, Z.Ya.}, \textit{On general boundary problems for equations of elliptic type}, 
Izv. Akad. Nauk SSSR Ser. Mat., \textbf{17}(1953), 6, 539–562.

\bibitem{Shef}
{Shefer, Yu.L.}, On a Transmission Problem Related to Models
of Electrocardiology, \textit{Journal of Siberian federal university. Math. and Physics}, 
\textbf{13}(2020), 5,  596--607.

\bibitem{Sh92a} 
{Shlapunov, A.A.}, 
{ On the Cauchy problem for the Laplace equation.} 
\textit{Siberian Math. Journal}, V.33, N. 3 (1992), p. 205--215.

\bibitem{Sh97}  
{Shlapunov, A.A.},
On the Cauchy Problem for some elliptic systems in a shell in $\mathbb {R}^n$, 
\textit{Zeitschrift f\"ur Angewandte Mathematik und Mechanik}, \textbf{77}(1997), 11, 849--859.

\bibitem{ShTaLMS} 
{Shlapunov, A.A., Tarkhanov, N.N.},  Bases with double orthogonality in the Cauchy problem for
systems with injective symbols, \textit{Proceedings LMS}, \textbf{3} (1995), 1, 1--52.

\bibitem{Simanca1987}  
S.Simanca, Mixed Elliptic Boundary Value Problems, \textit{Comm. in PDE}, \textbf{12}(1987), 
123-200.

\bibitem{Slob58}
{Slobodetskii, L. N.},
 \textit{Generalised spaces of S.L. Sobolev and their applications to boundary problems
          for partial differential equations},
   Science Notes of Leningr. Pedag. Institute \textbf{197} (1958), 54--112. 

\bibitem{Sm65}
{Smale, S.}, \textit{An infnite dimensional version of Sard's theorem},
Amer. J. Math. 87 (1965), no. 4, 861--866.

\bibitem{sol}
{Solonnikov, V.A.}, \textit{On boundary value problems for linear parabolic systems of differential equations of general form. Boundary value problems of mathematical physics. 
Part 3. On boundary 
value problems for linear parabolic systems of differential equations of general form},  
Trudy Mat. Inst. Steklov, {\bf 83}, 1965, 3--163 

\bibitem{2}   
{Sundnes, J., Lines,  G.T.,Cai,  X., Nielsen, B.F., Mardal, K.A., Tveito, A.},
\textit{Computing the Electrical Activity in the Heart}, Springer-Verlag, 2006.

\bibitem{svesh}
{Sveshnikov, A.G., Bogolyubov, A.N.,  Kravtsov, V.V.}, 
\textit{Lectures on mathematical physics}, M., Nauka, 2004. 

\bibitem{Tark95a}
{Tarkhanov, N.},  \textit{Complexes of Differential Operators}, 
Kluwer Academic Publishers, Dordrecht, NL, 1995.

\bibitem{Tark97}
{Tarkhanov, N.},  \textit{The Analysis of solutions of Elliptic Equations}, 
Kluwer Academic Publishers, Dordrecht, NL, 1997.

\bibitem{Tark36} 
{Tarkhanov, N.},  
\textit{The Cauchy Problem for Solutions of Elliptic Equations}, Akademie-Verlag, 
Berlin, 1995.

\bibitem{TikhArsX}
{Tikhonov, A.N., Arsenin, V.Ya.},   
\textit{Methods of Solving Ill-posed Problems},  Nauka, 
Moscow, 1986.



\end{thebibliography}
\end{document}